\documentclass[10pt]{amsart}
\usepackage[toc,page]{appendix}
\usepackage{amsfonts}
\usepackage{epic}
\usepackage{mathrsfs}
\usepackage{amsmath}
\usepackage{amssymb}
\usepackage{amscd}
\usepackage{amsthm}
\usepackage{enumerate}
\usepackage[dvips]{graphicx}
\usepackage[all,cmtip]{xy}
\usepackage{hyperref}
\usepackage{cases}
\usepackage{fancybox}
\usepackage{ulem}
\usepackage{color}
\usepackage[all]{xy}
\newcommand{\End}{{\rm End}}

\newcommand{\Hom}{{\rm Hom}}

\newtheorem{Thm}{Theorem}[section]
\newtheorem{Lem}[Thm]{Lemma}
\newtheorem{Def}[Thm]{Definition}
\newtheorem{Cor}[Thm]{Corollary}
\newtheorem{Prop}[Thm]{Proposition}
\newtheorem{Ex}[Thm]{Example}
\newtheorem{Rem1}[Thm]{Remark}

\newtheorem{Cor-Def}[Thm]{Corollary-Definition}
\newtheorem{fact}[Thm]{Fact}

\usepackage[margin=1.2in,centering]{geometry}
\title[Exceptional cycles for perfect complexes over gentle algebras]
{Exceptional cycles for perfect complexes \\ over gentle algebras}
\thanks{$^*$  \ Corresponding author}
\thanks{Supported by the NNSFC (National Natural Science Foundation of China) Grant No. 11971304.}
\author[Peng Guo, Pu Zhang]{Peng Guo, Pu Zhang$^*$}
\begin{document}

\begin{abstract} Exceptional cycles in a triangulated category $\mathcal T$ with Serre duality, introduced by N. Broomhead, D. Pauksztello, and D. Ploog, have a notable impact on the global structure of $\mathcal T$. In this paper we show that if $\mathcal T$ is homotopy-like, then any exceptional $1$-cycle is indecomposable and at the mouth; and any object in an exceptional $n$-cycle with $n\ge 3$ is at the mouth. Let $A$ be an indecomposable gentle $k$-algebra with $A\ne k$.
The Hom spaces between string complexes at the mouth are explicitly determined.
The main result classifies ``almost all" the exceptional cycles in $K^b(A\mbox{-}{\rm proj})$, using characteristic components and their AG-invariants, except those  exceptional $1$-cycles which are band complexes. Namely, the mouth of a characteristic component $C$ of $K^b(A\mbox{-}{\rm proj})$ forms a unique exceptional cycle in $C$, up to an equivalent relation $\approx$;
if the quiver of $A$ is not of type $A_3$, this gives all the exceptional $n$-cycle in $K^b(A\mbox{-}{\rm proj})$ with $n\ge 2$, up to $\approx$; and
a string complex is an exceptional $1$-cycle if and only if it is at the mouth of a characteristic component with  {\rm AG}-invariant $(1, m)$. However, a band complex at the mouth is possibly not an exceptional $1$-cycle.
\end{abstract}
\maketitle

\section{\bf Introduction}

\subsection{} An exceptional cycle in a triangulated category $\mathcal T$ with Serre duality has been introduced by N. Broomhead, D. Pauksztello, and D. Ploog [BPP].
It is a generalization of a spherical object (see e.g. [ST], [HKP1]), provides 
an invariant of triangle-equivalences, and closely relates to the global structure of $\mathcal T$.
Its importance also lies in the fact that an exceptional cycle induces an autoequivalences of $\mathcal T$ ([BPP, Theorem 4.5]), which is a generalization of tubular mutation in [M].

\vskip5pt

Gentle algebras, introduced by I. Assem and A. Skowronski [AS], have related to different topics in mathematics. It is closed under derived equivalence by J. Schr\"{o}er and A. Zimmermann [SZ], and exactly the class of finite-dimensional algebras $A$ such that the repetitive algebras $\hat A$ are special biserial ([AS], [PS]).
It appears in D. Vossieck's classification of algebras with discrete derived categories ([V]),
and in cluster tilted algebras (see e.g. [ABCP], [AG]). The combinatorial description of Auslander-Reiten triangles of $K^b(A\mbox{-}{\rm proj})$ has been given by G. Bobi\'nski \cite{B}. Recently,
a geometric derived realization and complete derived invariants of gentle algebras are given in [OPS], [APS] and [O].

\vskip5pt

The aim of this paper is to study exceptional cycles in an indecomposable homotopy-like triangulated category with Serre duality,
and to determine all the exceptional cycles in homotopy category $K^b(A\mbox{-}{\rm proj})$, where $A$ is an indecomposable finite-dimensional gentle algebra.

\subsection{} Throughout, $k$ is an algebraically closed field,
$\mathcal{T}$ is a $k$-linear Hom-finite Krull-Schmidt triangulated category with Serre functor (if no otherwise stated). Let $S: \mathcal{T}\longrightarrow\mathcal{T}$ be the right Serre functor. So,
for objects $X$ and $Y$ in $\mathcal T$, there is a $k$-linear isomorphism
$\Hom_\mathcal{T}(X, Y)\cong \Hom_\mathcal{T}(Y, S(X))^*,$
which is functorial in $X$ and $Y$, where $-^*$ is the $k$-dual $\Hom_k(-, k)$. Then $S$ is a triangle-equivalence ([Boc, Appendix]) and
$\mathcal{T}$ has Auslander-Reiten triangles with $S = \tau [1]$ on objects, where
$\tau$ is the Auslander-Reiten translate (see [RV, Theorem I. 2.4]).

\vskip5pt

An indecomposable object of $\mathcal T$ is said to be {\it at the mouth}, if the middle term of the Auslander-Reiten triangle  ending at it is indecomposable, or equivalently,
the middle term of the Auslander-Reiten triangle starting from it is indecomposable.

\vskip5pt

A triangulated category $\mathcal T$ will be said to be {\it homotopy-like}, if $M\ncong M[i]$ for any indecomposable object $M\in \mathcal{T}$ and  for all $i\ne 0$.
It is clear that $K^b(\mathcal B)$ is homotopy-like, where $\mathcal B$ is an arbitrary additive subcategory of an abelian category, which is closed under direct summands.
So $K^b(A\mbox{-}{\rm proj})$ is homotopy-like. On the other hand, if $\Lambda$ is a finite-dimensional self-injective algebra admitting a non-zero $\Omega$-periodic module, then the stable category
$\Lambda\mbox{-}\underline {\rm mod}$ is not homotopy-like.

\subsection{} Throughout, $A$ is an indecomposable finite-dimensional  gentle $k$-algebra.
Let $A$-mod be the category of finitely generated left $A$-modules, $\mathcal D^b(A)$ the bounded derived category of $A$-mod,
and  $K^b(A\mbox{-}{\rm proj})$  ($K^b(A\mbox{-}{\rm inj})$, respectively) the bounded homotopy category of finitely generated projective (injective, respectively) $A$-modules.
The Nakayama functor $(\Hom_A(-, A))^*$ induces componentwisely an equivalence $S: K^b(A\mbox{-}{\rm proj})\cong K^b(A\mbox{-}{\rm inj})$, so that for $P^\bullet\in K^b(A\mbox{-}{\rm proj})$
 and $X\in \mathcal D^b(A)$
there is a $k$-linear functorial isomorphism  in both arguments ([H1, p.37]): $$\Hom_{\mathcal D^b(A)}(P, X)\cong \Hom_{\mathcal D^b(A)}(X, S(P))^*.$$ Since $A$ is a Gorenstein algebra (\cite{GR}),
$K^b(A\mbox{-}{\rm proj}) = K^b(A\mbox{-}{\rm inj})$ in $\mathcal D^b(A)$ and $S$ is the Serre functor of $K^b(A\mbox{-}{\rm proj})$ ([H2]). So
$K^b(A\mbox{-}{\rm proj})$ is a $k$-linear, Hom-finite Krull-Schmidt triangulated category with Serre functor $S$, and having Auslander-Reiten triangles.

\subsection{} Let $d$ be an integer. An object $E\in\mathcal T$ is a $d$-Calabi-Yau object if $S(E)\cong E[d]$.
In general $d$-Calabi-Yau objects are not closed under taking direct summands (see [CZ]). For objects $X, Y\in \mathcal{T}$, let $\Hom^{\bullet}(X, Y)$ be the complex of $k$-vector spaces with
$\Hom^i(X, Y) = \Hom_{\mathcal{T}}(X, Y[i])$ and with zero differentials. Thus
$\Hom^{\bullet}(X, Y)=\bigoplus\limits_i \Hom_{\mathcal{T}}(X, Y[i])[-i]$.

\vskip5pt

By definition, {\it an exceptional $1$-cycle} in $\mathcal{T}$ is a $d$-Calabi-Yau object $E$ for some integer $d$, such that $\Hom^{\bullet}(E,E) \cong  k\oplus k[-d]$.
It is also called {\it a spherical object} for example in [HKP1] and [HKP2]. Here we use the name of an exceptional $1$-cycle in \cite{BPP}, for the unification of
an exceptional $n$-cycle with $n\ge 1$. Note that the terminology ``a spherical object" has (slight) different meanings in the literatures, see for example [ST], [KYZ], [CP].
For the definition of an exceptional $n$-cycle with $n\ge 2$ we refer to Subsection 2.1.

\vskip5pt

Let $E$ be an exceptional $1$-cycle which is $d$-Calabi-Yau. Then $d$ is unique. If $d\ne 0$, or $d = 0$ and $\End_\mathcal T(E) \cong k[x]/\langle x^2\rangle$ as algebras, then $E$ is indecomposable; if $d=0$ and $\End_\mathcal T(E) \cong k\times k$ as algebras, then $E$ is decomposable.
However we have

\begin{Thm}\label{indsphericalmouth}  Let $\mathcal T$ be an indecomposable $k$-linear Hom-finite Krull-Schmidt triangulated category with Serre functor. Assume that $\mathcal T$ is homotopy-like. Then

\vskip5pt

$(1)$ \ Any exceptional $1$-cycle is indecomposable, and at the mouth.

\vskip5pt

$(2)$ \ Any object in an exceptional $n$-cycle with $n\ge 3$ is at the mouth.
\end{Thm}

{\bf Remark} \ $(1)$ \ If $\mathcal T$ is not indecomposable, then an exceptional $1$-cycle may be decomposable and not at the mouth.
For example, $(k, k) = (k, 0)\oplus (0, k)$ is an exceptional $1$-cycle in $\mathcal D^b(k)\times \mathcal D^b(k)$, and it is not at the mouth.

$(2)$ \ An object in an exceptional $2$-cycle may be not at the mouth. For example, let $A$ be the path algebra of the quiver $1 \rightarrow 2 \rightarrow 3.$
Then $\mathcal D^b(A)$ has an exceptional $2$-cycle $(P(2), I(2))$, where $P(2)$ (respectively, $I(2)$) is the indecomposable projective (respectively, injective) $A$-module at vertex $2.$
However, $P(2)$ and $I(2)$ are not at the mouth of $\mathcal D^b(A)$.

\vskip5pt

The proof of Theorem \ref{indsphericalmouth} will be given in Section 3. The main ideas in the proof are
to use a special non-zero non isomorphism from an indecomposable object $M$ to $S(M)$, constructed in [RV] (see Lemma \ref{rv} below), and the mapping cone of the composition of morphisms in a homotopy cartesian square
(see Lemmas \ref{con1} and \ref{gfnot0} below).

\subsection{} To determine the exceptional cycles in $K^b(A\mbox{-}{\rm proj})$, we need
the notion of a characteristic component of $K^b(A\mbox{-}{\rm proj})$, to determine its shape, and to introduce its {\rm AG}-invariant.

\vskip5pt

An indecomposable object of $K^b(A\mbox{-}{\rm proj})$ is either a string complex or a band complex (see [BM]; also [B]).
The description of Auslander-Reiten triangles of $K^b(A\mbox{-}{\rm proj})$ (\cite[Main Theorem]{B}) shows that a connected component $C$ of $K^b(A\mbox{-}{\rm proj})$ either consists of string complexes, or consists of band complexes. It will be called {\it a characteristic component}, if $C$ contains a string complex at the mouth (thus $C$  consists of string complexes). By some results in [XZ], [Sch], [V], and [BR],
$C$ is of the form $$\Bbb ZA_n \ (n\ge 2), \ \ \ \ \Bbb ZA_{\infty}, \ \ \ \ \Bbb ZA_{\infty}/\langle \tau^n\rangle \ \ (n\ge 1).$$
Since up to shift there are only finitely many string complexes at
the mouth, by the shape of $C$, there exists a unique pair $(n, m)$ of integers, such that $\tau^{n}X\cong X[m-n]$, for any indecomposable object $X$ at the mouth of $C$. This pair $(n, m)$ will be called  Alaminos-Geiss invariant (or AG invariant in short) of $C$. For details see Section 4.

\vskip5pt

In \cite{AAG} a characteristic component and its invariant have been defined for $\hat{A}$-$\underline{\rm mod}$, where $\hat{A}$ is the repetitive algebra.
Since $A$ is Gorenstein, the Happel embedding $K^b(A\mbox{-}{\rm proj})\hookrightarrow \hat{A}\mbox{-}\underline{\rm mod}$ preserves the Auslander-Reiten components (\cite[Corollary\ 5.3]{HKR}).
Thus, a characteristic component and its AG invariant here coincide with the ones in \cite{AAG} (but we include
the component of type $\Bbb Z\vec{A_n}$).

\subsection{} Since objects in an exceptional $n$-cycle in $K^b(A\mbox{-}{\rm proj})$ are at the mouth (with a unique exception in the case $n=2$), the dimension of the Hom spaces between string complexes at the mouth
will play a central role in determining exceptional cycles in  $K^b(A\mbox{-}{\rm proj})$. This is given as follows.

\begin{Thm} \label{morphismbetweenforbiddenthreads} \ Let $A$ be  an indecomposable finite-dimensional gentle algebra,
$M$ and $M'$ string complexes at the mouth. Then
$$\dim_k\Hom_{K^b(A\mbox{-}{\rm proj})}(M, M')= \begin{cases} 2, & \ \ \ \ \mbox{if} \ M'\cong S(M)\cong M; \\
1, & \ \ \ \ \mbox{if} \ M' \ \mbox{is isomorphic to one of}  \ S(M) \ \mbox{and} \ M, \\
&  \ \ \ \ \mbox{but} \ M' \ \mbox{is not isomorphic to the other};
\\ 0, & \ \ \ \ \mbox{if} \ M'\ncong S(M)\ \mbox{and} \ M'\ncong M.\end{cases}$$
In particular, $$\dim_k \End(M)= \dim_k\Hom_{K^b(A\mbox{-}{\rm proj})}(M, S(M)) = \begin{cases} 2, & \ \ \ \ \mbox{if} \ M\cong S(M); \\ 1,
& \ \ \ \ \mbox{if} \ M\ncong S(M).\end{cases}$$
\end{Thm}

\begin{Cor} \label{orthogonal} \ Let $A$ be an indecomposable finite-dimensional gentle algebra, $C$ and $C'$ different characteristic components of $K^b(A\mbox{-}{\rm proj})$, up to shift.
Then $\Hom^\bullet(X, Y) = 0$ for $X\in C$ and $Y\in C'$.
\end{Cor}

\vskip5pt

The proof of Theorem \ref{morphismbetweenforbiddenthreads} and of Corollary \ref{orthogonal} will be given in Section 5.
The main tools used in the proof are Lemma \ref{rv}, the combinatorial description of morphisms between indecomposable objects in $\mathcal D^b(A)$, given by K. K. Arnesen, R. Laking and D. Pauksztello in \cite{ALP} (see Subsection 5.1), and the bijections between the set of permitted threads and the set of forbidden threads, given in [AAG] (see also [BB]. See Subsection 2.5).

\subsection{} The following main result classifies all the exceptional $n$-cycle in $K^b(A\mbox{-}{\rm proj})$ with $n\ge 2$. It turns out that such an exceptional cycle is exactly a truncation of the $\tau$-orbit of any indecomposable object at the mouth of a characteristic component, with few exceptions. For the equivalent relation $\approx$ on the set of exceptional cycles we refer to Subsection 2.2. 

\begin{Thm} \label{main} Let $A = kQ/I$ be a finite-dimensional gentle algebra with $A\ne k$, where $Q$ is a finite connected quiver such that the underlying graph of $Q$ is not of type $A_3$.

\vskip5pt

$(1)$ \ Let $C$ be a characteristic component of $K^b(A\mbox{-}{\rm proj})$ with  {\rm AG}-invariant $(n, m)$, and $X$ an indecomposable object at the mouth of $C$. Then
$$(X, \tau X, \cdots, \tau^{n-1}X)$$ is the unique exceptional cycle in $C$, up to the equivalent relation $\approx$.

\vskip5pt

$(2)$ \ Any exception $n$-cycle in $K^b(A\mbox{-}{\rm proj})$ with $n\ge 2$
is given in $(1)$, up to $\approx$.

\vskip5pt

$(3)$ \ Any object in an exceptional cycle in $K^b(A\mbox{-}{\rm proj})$ is indecomposable and at the mouth.

\vskip5pt

A string complex $E$ is an exceptional $1$-cycle if and only if $E$ is at the mouth of a characteristic component of  {\rm AG}-invariant $(1, d)$. In this case $E$ is a $d$-Calabi-Yau object.

\vskip5pt

If a band complex $E$ is an exceptional $1$-cycle, then $E$ is at the mouth of a homogenous tube. \end{Thm}

\vskip5pt

Note that N. Broomhead, D. Pauksztello, and D. Ploog \cite[5.1]{BPP} have pointed out 
the assertion $(1)$ for the gentle algebras $\Lambda (r, n, m)$ with $n>r$: these are exactly 
derived-discrete algebras of finite global dimension which is not of Dynkin type, up to derived equivalence.

\vskip5pt

{\bf Remaek} $(1)$ \ If $A = k$, then $K^b(A\mbox{-}{\rm proj}) = \mathcal D^b(k)$ has no mouths,
and $(k, k)$ is the unique exceptional cycle
in $K^b(A\mbox{-}{\rm proj})$, up to $\approx$. So, Theorem \ref{main}$(2)$ does not hold
for $A = k$.

$(2)$ \ If $A = kQ/I$ is a gentle algebra such that the underlying graph of $Q$ is of type $A_3$,
then $K^b(A\mbox{-}{\rm proj}) \cong \mathcal D^b(k(1\rightarrow 2\rightarrow 3))$, and $(P(3), I(3) = P(1), I(1), S(2))$ and $(P(2), I(2))$ are all the exceptional cycles
in $\mathcal D^b(k(1\rightarrow 2\rightarrow 3))$, up to $\approx$. So, Theorem \ref{main}$(1)$ and $(2)$ do not hold
in this case.

$(3)$ \ A band complex at the mouth of a homogeneous tube is not necessarily an exceptional $1$-cycle.  See Example \ref{bandatmouth}.

\vskip5pt

The proof of  Theorem \ref{main} will be given in Section 6. The main tools used in the proof are
Theorem \ref{morphismbetweenforbiddenthreads}, Lemma \ref{2cyclefornonA3}, and Theorem \ref{indsphericalmouth}.

\section{\bf Preliminaries}

\subsection{Exceptional cycles}

\begin{Def}\label{defnexcepcycle} \ {\rm(\cite{BPP})} \ An exceptional $n$-cycle in $\mathcal{T}$ with $n\ge 2$ is a sequence $(E_1, \cdots, E_n)$ of objects satisfying the following conditions$:$

\vskip5pt

${\rm (E1)}$ \ \ $\Hom^{\bullet}(E_i,E_i) \cong k$ for each $i;$

\vskip5pt

${\rm (E2)}$ \ \ there are integers $m_i$ such that $S(E_i)\cong E_{i+1}[m_i]$ for each  $i$, where $E_{n+1}:=E_1;$

\vskip5pt

${\rm (E3)}$ \ \ $\Hom^{\bullet}(E_i, E_j)=0$, unless $j=i$ or $j=i+1$. $($This condition vanishes if $n =2$.$)$

\end{Def}

The sequence $(m_1, \cdots, m_n)$ of integers in the definition is unique, and we will call $(E_1, \cdots, E_n)$ an exceptional cycle with respect to $(m_1, \cdots, m_n)$. It is clear that (see e.g. [GZ, Lemma 2.2]) a sequence $(E_1, \cdots, E_n)$ of objects in $\mathcal{T}$ with $n\ge 2$ is an exceptional cycle if and only if
it satisfies the condition ${\rm (E1')}$, ${\rm (E2)}$, and ${\rm (E3')}$, where

\vskip5pt

${\rm (E1')}$ \ \ $\Hom^{\bullet}(E_1, E_1) \cong k$;

\vskip5pt

${\rm (E3')}$ \ \ If $n\ge 3$, then $\Hom^{\bullet}(E_1, E_j)=0$ for $3\le j\le n$.

\vskip5pt

Each object in an exceptional $n$-cycle with $n\ge 2$ is indecomposable. If $(E_1, \cdots, E_n)$ is an exceptional cycle, then
$$(E_2, \ \cdots, E_n, \ E_1), \ \ \ \ (S(E_1), \cdots, S(E_n)), \ \ \ \  (\tau E_1, \cdots, \tau E_n)$$ are also exceptional cycles. If $(E_1, \cdots, E_n)$ is an exceptional cycle, with respect to $(m_1, \cdots, m_n)$, then for arbitrary integers $t_1, \cdots, t_n$, $$(E_1[t_1], \cdots, E_n[t_n])$$ is an exceptional cycle,  with respect to  $(m_1+t_1-t_2, \cdots, m_n+t_n-t_{n+1})$, where $t_{n+1} = t_1$.

\subsection{Equivalent relation $\approx$} \ Consider the set $\mathcal E_n$ of all the exceptional $n$-cycles in $\mathcal T$ with $n\ge 1$.
For $(E_1, \cdots, E_n)\in \mathcal E_n$ and $(E'_1, \cdots, E_n')\in \mathcal E_n$, define
$(E_1, \cdots, E_n)\approx (E'_1, \cdots, E_n')$ if and only if they are up to shift independently at each position and up to rotation, i.e., there are integers $t_1, \cdots, t_n, s$ with $0\le s \le n-1$, such that
$$(E'_1, \cdots, E_n') =  (E_{\sigma^s(1)}[t_1], \cdots, E_{\sigma^s(n)}[t_n])$$
where $\sigma$ is the cyclic permutation $(1 \ 2 \ \cdots \ n)$. Then $\approx$ is an equivalent relation on $\mathcal E_n$.

\vskip5pt

For $(E_1, \cdots, E_n)\in \mathcal E_n$, by ${\rm (E2)}$ one has
$(\tau E_1, \cdots, \tau E_n)\approx (S(E_1), \cdots, S(E_n))\approx (E_1, \cdots, E_n).$

\vskip5pt

The following fact is convenient for later use. See [GZ, Lemma 2.1].

\begin{Lem}\label{shift} \ $(1)$ \ If
$(E_1, \cdots, E_n)$ and $(E'_1, \cdots, E'_m)$ are exceptional cycles in $\mathcal{T}$ with $E'_i \cong E_1[t]$ for some integers $i$ and $t$, then $n = m$ and
$(E_1, \cdots, E_n)\approx (E'_1, \cdots, E'_m).$

Thus each indecomposable object occurs in at most one exceptional $n$-cycle with $n\ge 2$, up to the equivalent relation $\approx$.

In particular, an exceptional cycle can not be enlarged, i.e., if $(E_1, \cdots, E_n)$ is an exceptional $n$-cycle with $n\ge 1$, then there are no exceptional cycles of the form
$(E_1, \cdots, E_n, F_1, \cdots, F_t)$ with $t\ge 1$. Thus, an exceptional cycle can not be shorten such that it is again an exceptional cycle.

\vskip5pt

$(2)$ \ Assume that $(E_1, \cdots, E_n)$ is an exceptional cycle with $n\ge 3$. Then $E_i\ncong E_1[t]$ for all $2\le i\le n$ and for all $t\in\Bbb Z$.
\end{Lem}

\subsection{Gentle algebras} Let $Q = (Q_0, Q_1, s, e)$ be a finite connected quiver. Write the conjunction of paths from right to left. A bound quiver $(Q, I)$ algebra $A=kQ/I$ (see e.g. [ASS])  is {\it a gentle algebra} (see \cite{AS}),  if the following conditions are satisfied:

\vskip5pt

{(\rm G1)} \ for each vertex $x$, there are at most two arrows starting from $x$, and at most two arrows ending at $x$;

{(\rm G2)} \ for each arrow $\alpha$, there is at most one arrow $\beta$ with $\beta\alpha\notin I$, and at most one arrow $\gamma$ with $\alpha\gamma\notin I$;

{(\rm G3)} \ for each arrow $\alpha$, there is at most one arrow $\beta$ with $\beta\alpha\in I$, and at most one arrow $\gamma$ with $\alpha\gamma\in I$;

{(\rm G4)} \ the ideal $I$ is generated by paths of length 2.

\vskip5pt

A gentle algebra is possibly infinite dimensional. We only consider finite-dimensional gentle algebra.
We need the following equivalent definition of a gentle algebra (see \cite{BR}, \cite{B}).  A bound quiver algebra $A=kQ/I$ is gentle, where $I$ is generated by paths of length $2$, if there are maps $s', e': Q_1\rightarrow \{1, -1\}$ such that

\vskip5pt

(i) \ if $\alpha\in Q_1$ and $\beta\in Q_1$ start at the same vertex with $\alpha\ne \beta$, then $s'(\alpha)=-s'(\beta)$;

(ii) \ if $\alpha\in Q_1$ and $\beta\in Q_1$ end at the same vertex with $\alpha\ne \beta$, then $e'(\alpha)=-e'(\beta)$;

(iii) \  if $\alpha\in Q_1$ ends at the vertex where $\beta\in Q_1$ starts and $\beta\alpha\notin I$, then $s'(\beta) = - e'(\alpha)$;

(iv) \  if $\alpha\in Q_1$ ends at the vertex where $\beta\in Q_1$ starts and $\beta\alpha\in I$, then $s'(\beta) =e'(\alpha)$.

\vskip5pt

In this case, for a path $p=\alpha_n\cdots\alpha_1$ with each $\alpha_i\in Q_1$, we define $s'(p):=s'(\alpha_1)$ and $e'(p):=e'(\alpha_n)$.

\subsection{Permitted (forbidden, respectively) threads of a gentle algebra} Let $A = kQ/I$ be a gentle algebra. A path $p$ is {\it a permitted path} if $p\notin Q_0$ and $p\notin I$.
Following [AAG], a {\it non-trivial permitted thread} $p$ is a permitted path, such that $p\alpha  = 0 = \alpha p$ for all $\alpha\in Q_1$.

\vskip5pt

A {\it trivial permitted thread} is a vertex $x$ such that

\vskip5pt

(i) \ there is at most one arrow starting from $x$, and at most one arrow ending at $x$;

(ii) \ If $\beta$ is an arrow ending at $x$  and
$\gamma$ is an arrow starting from $x$, then $\gamma\beta\notin I$.

\vskip5pt

Denote by $\mathcal{P}_A$ the set of permitted threads (whatever they are non-trivial or trivial).

\vskip5pt

A path $p=\alpha_n\cdots\alpha_1$ with each $\alpha_i\in Q_1$ is {\it a forbidden path}, if $n\ge 1$ and $\alpha_1, \cdots, \alpha_n$ are pairwise different, such that $\alpha_{i+1}\alpha_i\in I$ for $1\leq i\leq n-1$.  {\it A non-trivial forbidden thread} $p$ is a forbidden path, if for all $\alpha\in Q_1$, neither $\alpha p$ nor $p\alpha$ is a forbidden path.

\vskip5pt

{\it A trivial forbidden thread} is a vertex $x$, such that

\vskip5pt

(i) \ there is at most one arrow starting from $x$, and at most one arrow ending at $x$;

(ii) \ If $\beta$ is an arrow ending at $x$  and
$\gamma$ is an arrow starting from $x$, then $\gamma\beta\in I$.

\vskip5pt

Denote by $\mathcal{F}_A$ the set of forbidden threads (whatever they are non-trivial or trivial).

\vskip5pt

Extend the maps $s'$ and $e'$ to trivial permitted (forbidden, respectively) threads as follows. We write a vertex $x$ as ${\bf1}_x$. For a trivial permitted thread $x\in Q_0$, since $Q$ is connected, there is either $\gamma\in Q_1$ with $s(\gamma)=x$, or $\beta\in Q_1$ with $e(\beta)=x$.
Define
$$s'({\bf1}_x)=-e'({\bf1}_x)=-s'(\gamma), \ \ \mbox{or} \ \ s'({\bf1}_x)=-e'({\bf1}_x)=e'(\beta).$$ For a trivial forbidden thread $y\in Q_0$, there is either $\gamma\in Q_1$ with $s(\gamma)=y$, or $\beta\in Q_1$ with $e(\beta)=y$. Define $$s'({\bf1}_y)=e'({\bf1}_y)=-s'(\gamma), \ \ \mbox{or} \ \ s'({\bf1}_y)=e'({\bf1}_y)=-e'(\beta).$$

\subsection{Bijection between $\mathcal{P}_A$ and $\mathcal{F}_A$} Avella-Alaminos and Geiss [AAG] (see also [BB]) observed that there are bijections between
$\mathcal{P}_A$ and $\mathcal{F}_A$. For each permitted thread $v$, there is a unique forbidden thread $w$ such that
$$e(w)=e(v), \ \ e'(w)=-e'(v).$$ This defines a bijection $\Phi_1: \mathcal{P}_A\rightarrow \mathcal{F}_A$, $v\mapsto w$.
Also, for each forbidden thread $w$, there is a unique permitted thread $v$ such that
$$s(w)=s(v), \ \ s'(w)=-s'(v).$$ This defines a bijection $\Phi_2: \mathcal{F}_A\rightarrow \mathcal{P}_A$, $w\mapsto v$.
Note that $\Phi_2\Phi_1 : \mathcal{P}_A \rightarrow \mathcal{P}_A $ is not necessarily the identity map.

\subsection{Homotopy strings and homotopy bands}
For each $\alpha\in Q_1,$ define its formal inverse $\alpha^{-1}$ such that $s(\alpha^{-1}) = e(\alpha), \ e(\alpha^{-1}) = s(\alpha)$ and $(\alpha^{-1})^{-1} = \alpha$. For path $p = \alpha_n\cdots\alpha_1$ with each $\alpha_i\in Q_1$, define $p^{-1}:=\alpha_1^{-1}\cdots\alpha_n^{-1}$. Define
$$s(p^{-1}):=e(p), \ e(p^{-1}):=s(p), \ s'(p^{-1}):=e'(p), \ e'(p^{-1}):=s'(p).$$

A {\it homotopy letter} $w$ is either a permitted path $p$ (in this case, $w$ is said to be {\it direct}), or the formal inverse $p^{-1}$ of $p$ (in this case, $w$ is {\it inverse}). The composition $ww'$ of homotopy letters is defined if $e(w')=s(w)$ and if the following conditions are satisfied:

\vskip5pt

(i) \  if both  $w$ and $w'$ are direct, or, if both  $w$ and $w'$ are inverse, then $e'(w')=s'(w);$

(ii) \ if one of $w$ and $w'$ is direct and the other is inverse, then $e'(w')=-s'(w)$.

\vskip5pt

A non-trivial {\it homotopy string} is a sequence of consecutive composable homotopy letters
$w=w_n\cdots w_1.$ By the definition, a non-trivial permitted (forbidden, respectively) thread is a non-trivial homotopy string.

\vskip5pt
For $x\in Q_0$ and $\varepsilon\in\{1,-1\}$, define two {\it trivial homotopy strings} ${\bf1}_{x,1}$ and ${\bf1}_{x,-1}$, with $s({\bf1}_{x,\varepsilon}):=x=:e({\bf1}_{x,\varepsilon})$ and $s'({\bf1}_{x,\varepsilon})=\varepsilon, \ e'({\bf1}_{x,\varepsilon})=-\varepsilon$ . By the definition, a trivial permitted (forbidden, respectively) thread can be regarded a
trivial homotopy string.

\vskip5pt

Put $({\bf1}_{x,\varepsilon})^{-1}={\bf1}_{x,-\varepsilon}$. For the composition of  (non-trivial or trivial) homotopy strings we refer to \cite [Section 3]{B}.
Extend the maps $s, e, s', e'$ to homotopy strings as
$$s(w)=s(w_1),  \ \ e(w)=e(w_n), \ \ s'(w)=s'(w_1), \ \ e'(w)=e'(w_n).$$
{\it The degree} deg$(w)$ of a  homotopy string $w$ is the number of direct homotopy letters in $w$ minus the number of inverse homotopy letter in $w$. For examples,
in the path algebra of the quiver $
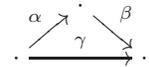
\xymatrix@R=0.3cm@C=0.5cm{&\cdot\ar[dr]^-{\beta} \\\cdot \ar[ur]^\alpha \ar[rr]^-\gamma && \cdot}$ the degree of the homotopy string $\gamma^{-1}\beta\alpha$ is $0$, rather than $1$;
and in the algebra given by the same quiver with relation $\beta\alpha$, the degree of the homotopy string $\gamma^{-1}\beta\alpha$ is $1$, rather than $0$.

\vskip5pt

A non-trivial homotopy string $w=w_n\cdots w_1$ is a {\it homotopy band} if the conditions are satisfied:
\vskip5pt

(i) \  deg$(w)=0$ and $s(w)=e(w)$;

(ii) \ one of $w_n$ and $w_1$ is direct and the other is inverse; and

(iii) \ $w$ is not a proper power of a homotopy string.

\subsection{String complexes and band complexes} Bobi\'nski \cite{B} has introduced a string complex and a band complex in $K^b(A\mbox{-}{\rm proj})$ for a gentle algebra $A$ (see also [BM] for
$\mathcal D^b(A)$; also [AAG] and [BR]).
For a homotopy string $w$ and an integer $m$, there is an associated {\it string complex} $P_{m, w}$.
Also, for a homotopy band $w$, an integer $m$, and an indecomposable automorphism $\mu$ of a finite-dimensional vector space, there is an associated {\it band complex}
$P_{m, w, \mu}$. For details we refer to \cite[Section 3]{B} and \cite[4.1]{BM}. Note that for different integers $m$ and $m'$, $P_{m', w}$ is a shift of $P_{m, w}$, and
$P_{m', w, \mu}$ is a shift of $P_{m, w, \mu}$.

\begin{Thm}\label{sorb} {\rm (\cite{B}, \cite{BM})} Let $A$ be a gentle algebra. Then up to isomorphism any indecomposable object in $K^b(A\mbox{-}{\rm proj})$
is either a string complex $P_{m,w}$,  or a band complex $P_{m,w,\mu}$.
\end{Thm}

A string complex given by a forbidden thread is of the following important property.

\begin{Lem}\label{mouth} {\rm (\cite{B}, Corollary 6.3)} \ Let $A$ be a gentle algebra, $w$ a homotopy string, and $m$ an integer. Then the string complex $P_{m,w}$ is at the mouth if and only if $w$ is a forbidden thread.
\end{Lem}

\subsection{Auslander-Reiten triangles in $K^b(A\mbox{-}{\rm proj})$}
For each homotopy string $w$, in order to describe the Auslander-Reiten triangle involving the string complex $P_{m,w}$,
Bobi\'nski \cite{B} defines $w^{+}$, $w_{+}$ and $w^{+}_{+}$, and integer $m'(w), m''(w)$,
where $w^{+}$ and $w_{+}$ are either homotopy strings or $0$, and $w^{+}_{+}$ is a homotopy string.
Also, for an indecomposable automorphism $\mu$ of a finite-dimensional non-zero $k$-vector space $V$, one can define $\mu_1$ and $\mu_2$,
which are indecomposable automorphisms of the associated finite-dimensional $k$-vector spaces. See \cite{B} for details.

\begin{Thm}\label{2terms}$($\cite[{\rm Main theorem}]{B}$)$ Let $A$ be an indecomposable
finite-dimensional gentle $k$-algebra with $A\ne k$, $m$ an integer.

\vskip5pt

$(1)$  \ Let $w$ be a homotopy string. Then there is an Auslander-Reiten triangle in $K^b(A\mbox{-}{\rm proj}):$
$$P_{m,w}\longrightarrow P_{m+m'(w), w_{+}}\oplus P_{m, w^{+}}\longrightarrow P_{m+m''(w), w^{+}_{+}}\longrightarrow P_{m,w}[1]$$
consisting of string complexes, where if $w_+ = 0$ then $P_{m+m'(w), w_+} = 0$, and if  $w^+ = 0$ then $P_{m, w^+} = 0$.
\vskip5pt

$(2)$ \ Let $w$ be a homotopy band, and $\mu$ an indecomposable automorphism of a finite-dimensional vector space. Then there is an Auslander-Reiten triangle in $K^b(A\mbox{-}{\rm proj}):$
$$P_{m,w,\mu}\longrightarrow P_{m,w,\mu_1}\oplus P_{m,w,\mu_2}\longrightarrow P_{m,w,\mu}\longrightarrow P_{m,w,\mu}[1]$$
consisting of band complexes, where if $\mu_1 = 0$ then $P_{m, w, \mu_1} = 0$. In particular, a band complex is in a homogeneous tube.
\end{Thm}

We stress that in Theorem \ref{2terms}$(1)$ the middle can not be zero, i.e., the situation $w^{+} = 0 = w_{+}$ can not occur (since $A\ne k$ by our assumption, this follows from Lemmas \ref{rv} and \ref{indec}), thus the middle can not be zero. The same remark on Theorem \ref{2terms}$(2)$.

\vskip5pt

Note that band complex $P_{m,w,\mu}$ is at the mouth if and only if $\mu$ is an indecomposable automorphism of $1$-dimensional vector space (i.e. $\mu\in k^{*}$).
By Theorem \ref{2terms}, the number of indecomposable direct summands of the middle terms of an Auslander-Reiten triangle is $1$ or $2$; and by Theorems \ref{sorb} and \ref{2terms}, a component of the Auslander-Reiten quiver of $K^b(A\mbox{-}{\rm proj})$ consists of either string complexes, or band complexes.

\subsection{Mapping cones of irreducible maps between string complexes} The following fact is known in \cite[p.38]{BGS}.
In $\Lambda$-mod of artin algebra $\Lambda$, there is a corresponding result (see \cite{Br}).

\begin{Lem}\label{cone2} \ Let $f$ be an irreducible map between string complexes in $K^b(A\mbox{-}{\rm proj})$. Then ${\rm Cone}(f)$ is at the mouth.
\end{Lem}

\section{\bf Proof of Theorem \ref{indsphericalmouth}}

\subsection{} The following observation is essentially due to I. Reiten and M. Van den Bergh [RV].

\begin{Lem}\label{rv} Let $\mathcal{T}$ be an indecomposable
$k$-linear Hom-finite Krull-Schmidt triangulated category with Serre functor. Suppose that $\mathcal T$ is homotopy-like and $\mathcal T\ncong \mathcal D^b(k)$.
Then for any indecomposable object $M$ of $\mathcal T$,
there are no Auslander-Reiten triangles of the form $\tau M\longrightarrow 0 \longrightarrow M \longrightarrow S(M);$
and there is a non-zero morphism  $f: M \longrightarrow S(M)$ which is not an isomorphism. \end{Lem}

{\bf Remark.} \ In $\mathcal D^b(k)$ any non-zero morphism  $k \longrightarrow S(k) = k$
is an isomorphism, and $k[-1] \longrightarrow 0 \longrightarrow k \stackrel {{\rm Id}_k} \longrightarrow k$ is an Auslander-Reiten triangle.

\vskip5pt

\noindent {\bf Proof of Lemma \ref{rv}.} \ Otherwise, assume
that $\tau M \longrightarrow 0 \longrightarrow M \stackrel {h} \longrightarrow S(M)$ is an Auslander-Reiten triangle.
Suppose that $g: X\longrightarrow M$ is an arbitrary morphism with $X$ indecomposable and $X\ncong M$.
Then $g$ factors through $0$, thus $g = 0$. So $\Hom_{\mathcal T}(X, M) = 0$.
Since $h: M \longrightarrow S(M)$ is an isomorphism, $\tau M = S(M)[-1] \cong M[-1]$. So $M \longrightarrow 0 \longrightarrow M[1] \stackrel {h[1]} \longrightarrow M[1]$
is an Auslander-Reiten triangle. Similarly, $\Hom_{\mathcal T}(M, Y) = 0$ for an
arbitrary indecomposable object $Y$ with $Y\ncong M$.

\vskip5pt

Let $\langle M\rangle$ be the smallest triangulated subcategory of $\mathcal T$ containing $M$. Since by assumption $M\ncong M[i]$ for
each $i\ne 0$, it follows that $\Hom_\mathcal T(M, M[i]) = 0 = \Hom_{\mathcal T}(M[i], M)$ for all $i\ne 0$. Since $\tau (M) \longrightarrow 0 \longrightarrow M \stackrel {h} \longrightarrow S(M)$ is an Auslander-Reiten triangle, any non-isomorphism $M\longrightarrow M$ factors through $0$, i.e., $\End(M)$ is a field, and hence $\End(M) =k$, since $k$ is algebraically closed.  All together one has
$$\Hom_\mathcal T(M[i], M[j]) = 0, \ \ \forall \ i \ne j, \ \ \ \End(M[i]) = k, \ \ \forall \ i\in\Bbb Z.$$
By the construction of
$\langle M\rangle$ (see [Ro, 3.1]) one sees that $\langle M\rangle$ is exactly the full subcategory of $\mathcal T$
consisting of finite direct sums of objects of the form $M[n]$ with $n\in\Bbb Z$. Consider the functor $\mathcal D^b(k)\rightarrow \langle M\rangle $ given by
$$k[n]\mapsto M[n], \ \ \  \lambda {\rm Id}_{k[n]} \mapsto \lambda {\rm Id}_{M[n]}, \ \ \ \forall \ n\in\Bbb Z, \ \lambda\in k.$$ This gives the triangle-equivalence
$\mathcal D^b(k)\cong \langle M\rangle$.

\vskip5pt

Let  $\mathcal T'$ be the smallest triangulated subcategory of $\mathcal T$ containing all the indecomposable objects which are not isomorphic to $M[n]$ for all $n\in\Bbb Z$.
Then $\Hom_\mathcal T(\langle M\rangle, \mathcal T') = 0 = \Hom_\mathcal T(\mathcal T', \langle M\rangle)$.
Since by assumption $\mathcal T$ is Krull-Schmidt, it follows that any object of $\mathcal T$ is a direct sum $M'\oplus X$ with $M'\in \langle M\rangle$ and $X\in \mathcal T'$.
Thus $\mathcal T\cong \langle M\rangle \times \mathcal T'$. Since by assumption $\mathcal{T}$ is indecomposable, $\mathcal T\cong \langle M\rangle\cong \mathcal D^b(k)$, which contradicts the assumption. This proves the first assertion.

\vskip5pt

Now we show the second assertion. Since $\End(M)$ is a local algebra and $k$ is algebraically closed, $\End(M)/{\rm rad}\End(M) \cong k.$ Put
$\theta: \End(M)\longrightarrow k$ to be the canonical surjective map with ${\rm Ker}\theta = {\rm rad}\End(M)$. Then $\theta\in(\End(M))^*$.
Consider the $k$-linear isomorphism $\eta: \Hom_{\mathcal{T}}(M, S(M))\longrightarrow (\End(M))^*,$ and let $f: M\longrightarrow S(M)$ be the non-zero morphism
such that $\eta(f) = \theta$. Embedding $f$ into a distinguished triangle
$$\tau M \longrightarrow Z \longrightarrow M \stackrel {f} \longrightarrow S(M) = (\tau M)[1].$$
By the proof of [RV, Theorem I.2.4], it is an Auslander-Reiten triangle. Thus $f$ is not an isomorphism (otherwise $\tau M \longrightarrow 0 \longrightarrow M \stackrel {f} \longrightarrow S(M)$ is an Auslander-Reiten triangle).
\hfill $\square$

\subsection{} Applying Lemma \ref{rv} one can prove

\begin{Lem}\label{indspherical}  Let $\mathcal T$ be an indecomposable
$k$-linear Hom-finite Krull-Schmidt triangulated category with Serre functor. Assume that $\mathcal T$ is homotopy-like.
Then any exceptional $1$-cycle in $\mathcal T$  is indecomposable.
\end{Lem}

\begin{proof} \ Since $\mathcal D^b(k)$ has no exceptional $1$-cycles, one may assume that $\mathcal T\ncong \mathcal D^b(k).$

Assume that $E\cong E_1\oplus E_2$ is an exceptional $1$-cycle with $E_1 \ne 0\ne E_2$. Then $E$ has to be a $0$-Calabi-Yau object, so $S(E)\cong E$ and
$\Hom^{\bullet}(E, E)\cong k\oplus k$. In particular, $\Hom_{\mathcal{T}}(E_1, E_1)\cong k \cong \Hom_{\mathcal{T}}(E_2, E_2)$
and $\Hom_{\mathcal{T}}(E_1, E_2)=0$. It follows that $E_1$ and $E_2$ are indecomposable.
Since $\mathcal T$ is Krull-Schmidt and $E_1\oplus E_2\cong S(E_1)\oplus S(E_2)$, either $E_1\cong S(E_1)$ or $E_2\cong S(E_1)$.

If $E_1\cong S(E_1)$, then by Lemma \ref{rv}, there exists a non-isomophism $0\ne f: E_1\longrightarrow S(E_1)\cong E_1$. This contradicts $\Hom_{\mathcal T}(E_1, E_1)\cong k$.

If $E_2\cong S(E_1)$, then one gets a contradiction $k\cong \Hom_{\mathcal T}(E_1, E_1)\cong \Hom(E_1, S(E_1))\cong \Hom_{\mathcal T}(E_1, E_2)$ $=0$.
This completes the proof. \end{proof}

\subsection{} The following general result in triangulated category is a consequence of A. Neemann [N, Lemma 1.4.4].
\begin{Lem}\label{con1} Let $\mathcal T$ be a triangulated category, and
$X\xrightarrow{\binom{f_1}{f_2}} Y_1\oplus Y_2\xrightarrow{(g_1,-g_2)} Z\longrightarrow X[1]$ a distinguished triangle in $\mathcal T$. Then
${\rm Cone}(f_1)\cong {\rm Cone}(g_2)$, ${\rm Cone}(f_2)\cong {\rm Cone}(g_1)$, ${\rm Cone}(g_1f_1)\cong {\rm Cone}(f_1)\oplus {\rm Cone}(g_1)$.
\end{Lem}

\begin{proof} Applying [N, Lemma 1.4.4] to the homotopy cartesian square

$$\xymatrix{X\ar[r]^{f_1} \ar[d]_-{f_2} & Y_1\ar[d]^{g_1}
\\ Y_2\ar[r]^{g_2} & Z}$$ one has ${\rm Cone}(f_1)\cong {\rm Cone}(g_2)$ and ${\rm Cone}(f_2)\cong {\rm Cone}(g_1)$, such that there is a commutative diagram of distinguished triangles:
$$\xymatrix{X\ar[r]^{f_2} \ar[d]_-{f_1} & Y_2\ar[r] \ar[d]^{g_2}& {\rm Cone}(f_2)\ar[r]^-a\ar[d]_-{\cong}^-h& X[1]\ar[d]^-{f_1[1]}&
\\ Y_1\ar[r]^{g_1} & Z\ar[r] & {\rm Cone}(g_1)\ar[r]^-b& Y_1[1]}$$

\vskip5pt

\noindent Thus $bh = f_1[1]a$. Applying the octahedral axiom one gets the commutative diagram

$$\xymatrix@R=0.6cm{X\ar[r]^-{f_1}\ar@{=}[d]& Y_1\ar[r]^-c\ar[d]_-{g_1}& {\rm Cone}(f_1)\ar[r]\ar@{-->}[d]& X[1]\ar@{=}[d] &&
\\X\ar[r]^-{g_1f_1}  & Z\ar[r]\ar[d] & {\rm Cone}(g_1f_1)\ar[r]\ar@{-->}[d]  & X[1]\ar[d]^-{f_1[1]} &&
\\ &  {\rm Cone}(g_1)\ar@{=}[r]\ar[d]^-b   &  {\rm Cone}(g_1)\ar[r]^-b\ar@{-->}[d]^-w   & Y_1[1] &&
\\ &Y_1[1]\ar[r]^-{c[1]}& {\rm Cone}(f_1)[1] &&&}$$

\vskip5pt

\noindent The third column gives the distinguished triangle
$${\rm Cone}(f_1)\longrightarrow {\rm Cone}(g_1f_1)\longrightarrow {\rm Cone}(g_1)\stackrel {w}\longrightarrow {\rm Cone}(f_1)[1].$$ Since
$w=c[1]b=c[1]f_1[1] ah^{-1} = (cf_1)[1]ah^{-1} = 0$,  this distinguished triangle splits, and hence
${\rm Cone}(g_1f_1)\cong {\rm Cone}(f_1)\oplus {\rm Cone}(g_1)$. \end{proof}

\subsection{}  The following lemma will play an important role in this paper.

\begin{Lem} \label{gfnot0}  \ Let $\mathcal T$ be a $k$-linear Hom-finite Krull-Schmidt triangulated category with Serre functor.
Suppose that $\tau E\xrightarrow{\binom{f_1}{f_2}} Z_1\oplus Z_2\xrightarrow{(g_1,-g_2)} E\xrightarrow{\sigma} S(E)$ is an Auslander-Reiten triangle of $\mathcal T$ with $Z_1$ indecomposable and $Z_2\ne 0$.
If one of the following conditions is satisfied, then $g_1f_1: \tau E\longrightarrow E$ is not zero$:$

\vskip5pt

$(1)$ \ $\Hom_\mathcal T(E, E)\cong k$.

\vskip5pt

$(2)$ \ $\Hom_\mathcal T(E, E)\cong k\oplus k$ and $S(E)\cong E$.
\end{Lem}

\begin{proof} \ According to Lemma \ref{con1}, ${\rm Cone}(g_1f_1)\cong {\rm Cone}(f_1)\oplus {\rm Cone}(g_1)$. We claim that $g_1f_1: \tau E\longrightarrow E$ is not zero. Otherwise, the distinguished triangle ${\rm Cone}(g_1f_1)[-1]\longrightarrow \tau E\xrightarrow{g_1f_1=0}E\longrightarrow {\rm Cone}(f_1)\oplus {\rm Cone}(g_1)$ splits, and  hence ${\rm Cone}(f_1)\oplus {\rm Cone}(g_1)\cong E\oplus S(E)$. Note that ${\rm Cone}(f_1)\ne 0$ (otherwise, $f_1$ is an isomorphism, which contradicts that $f_1$ is an irreducible map).
Similarly, ${\rm Cone}(g_1)\ne 0$. Since $\mathcal T$  is Krull-Schmidt, it follows that either ${\rm Cone}(f_1)\cong E$ or ${\rm Cone}(g_1)\cong E$.

\vskip5pt

$(1)$ \ Assume that $\Hom_\mathcal T(E, E)\cong k$.

\vskip5pt

If ${\rm Cone}(f_1)\cong E$, then one has distinguished triangles
$$\xymatrix{\tau E\ar[r]^-{f_1} & Z_1\ar[r] & E\ar[r]^-{\sigma_1}& S(E) \\ \tau E\ar[r]^-{{\binom{f_1}{f_2}}} & Z_1\oplus Z_2\ar[r]^-{(g_1,-g_2)} & E\ar[r]^-\sigma & S(E).}$$
Note that $\sigma_1$ and $\sigma$ are linearly independent (otherwise, the two distinguished triangles above are isomorphic, and then  $Z_1\cong Z_1\oplus Z_2$, which is absurd).
Thus ${\rm dim}_k\Hom_\mathcal T(E, S(E))\ge 2$. This contradicts ${\rm dim}_k\Hom_\mathcal T(E, S(E)) = {\rm dim}_k\Hom_\mathcal T(E, E) = 1$.

\vskip5pt

If ${\rm Cone}(g_1)\cong E$, then one gets a distinguished triangle $Z_1\stackrel {g_1}\longrightarrow E\stackrel {s}\longrightarrow E \longrightarrow Z_1[1]$. Note that $s\ne 0$ (otherwise, we get a contradiction $Z_1 \cong E[-1]\oplus E$). Thus $s$ and ${\rm Id}_E$ are linearly independent (otherwise, $s$ is an isomorphism, and then one gets a contradiction $Z_1 = 0$).
So ${\rm dim}_k\Hom_\mathcal T(E, E)\ge 2$, which contradicts $\Hom_\mathcal T(E, E)\cong k$.

\vskip5pt

All together $g_1f_1\ne 0$.

\vskip5pt

$(2)$ \ Assume that $\Hom_\mathcal T(E, E)\cong k\oplus k$ and $S(E)\cong E$, say, with isomorphism $h: E\longrightarrow S(E)$.

\vskip5pt

If ${\rm Cone}(f_1)\cong E$, then as in the proof of $(1)$ one has distinguished triangles
$$\xymatrix{\tau E\ar[r]^-{f_1} & Z_1\ar[r] & E\ar[r]^-{\sigma_1}& S(E) \\ \tau E\ar[r]^-{{\binom{f_1}{f_2}}} & Z_1\oplus Z_2\ar[r]^-{(g_1,-g_2)} & E\ar[r]^-\sigma & S(E)}$$
such that  $h^{-1}\sigma_1$ and $h^{-1}\sigma$ are linearly independent in ${\rm rad}{\rm End}_\mathcal T(E)$. Since ${\rm End}_\mathcal T(E)$ is local, it follows that $h\sigma_1$, $h\sigma$, ${\rm Id}_E$ are linearly independent in ${\rm End}_\mathcal T(E)$, which contradicts ${\rm dim}_k\Hom_{\mathcal T}(E, E) = 2$.

\vskip5pt

If ${\rm Cone}(g_1)\cong E$, then as in the proof of $(1)$ one gets a distinguished triangle $Z_1\stackrel {g_1}\longrightarrow E\stackrel {s}\longrightarrow E \longrightarrow Z_1[1]$, with $s\ne 0$.
By the two distinguished triangles
$$\xymatrix{E[-1]\ar[r]& Z_1\ar[r]^-{g_1} & E\ar[r]^-{s}& E \\ E[-1]\ar[r]^-{{\binom{f_1h[1]}{f_2h[1]}}} & Z_1\oplus Z_2\ar[r]^-{(g_1,-g_2)} & E\ar[r]^-{h^{-1}\sigma} & E}$$
one knows that $s$ and $h^{-1}\sigma$ are linearly independent in ${\rm rad}{\rm End}_\mathcal T(E)$. Since ${\rm End}_\mathcal T(E)$ is local, it follows that $s$, $h^{-1}\sigma$, ${\rm Id}_E$ are linearly independent in ${\rm End}_\mathcal T(E)$, again contradicts ${\rm dim}_k\Hom_{\mathcal T}(E, E) = 2$.

\vskip5pt

All together  $g_1f_1\ne 0$. This completes the proof. \end{proof}

\subsection{}  To prove Theorem \ref{indsphericalmouth}, we first show the following

\begin{Lem} \label{sphericalmouth}  \ Let $\mathcal T$  be an indecomposable $k$-linear Hom-finite
Krull-Schmidt triangulated category with Serre functor. Assume that $\mathcal T$ is homotopy-like. Then any exceptional $1$-cycle in $\mathcal{T}$ is at the mouth.
\end{Lem}

\begin{proof} \ Since $\mathcal D^b(k)$ has no exceptional $1$-cycles, one may assume that $\mathcal T\ncong \mathcal D^b(k).$

Let $E$ be an exceptional $1$-cycle in $\mathcal{T}$, which is a $d$-Calabi-Yau object.
By Lemma \ref{indspherical}, $E$ is indecomposable. Assume that $E$ is not at the mouth, i.e., the middle term of the Auslander-Reiten triangle ending at $E$ is either $0$, or of the form
$Z_1\oplus Z_2$ with $Z_1$ indecomposable and $Z_2\ne 0$. However, the first case is impossible, by Lemma \ref{rv}. Thus there is an Auslander-Reiten triangle
$$\tau E\xrightarrow{\binom{f_1}{f_2}} Z_1\oplus Z_2\xrightarrow{(g_1,-g_2)} E\xrightarrow{\sigma} S(E)$$
with $Z_1$ indecomposable and $Z_2\ne 0$. Then $g_1f_1: \tau E\longrightarrow E$ is not zero. In fact, if $d\ne 0$, then $\Hom_{\mathcal T}(E, E)\cong k$, and hence $g_1f_1\ne 0$, by Lemma \ref{gfnot0}$(1)$; if $d =0$, then $\Hom_\mathcal T(E, E)\cong k\oplus k$ and $S(E)\cong E$, and hence $g_1f_1\ne 0$, by Lemma \ref{gfnot0}$(2)$. It is clear that $g_1f_1$ is not an isomorphism (otherwise, $f_1$ is a splitting monomorphism, which contradicts that $f_1$ is irreducible). Now we divide into $2$ cases.

\vskip5pt

If $d = 1,$ then $\tau E \cong S(E)[-1]\cong E$ and $g_1f_1\in\Hom_{\mathcal T}(E, E)\cong k$. Since $g_1f_1$ and ${\rm Id}_E$ are linearly independent, this contradicts $\Hom_{\mathcal T}(E, E)\cong k$.

\vskip5pt

If $d\ne 1$, then $\tau E \cong S(E)[-1]\cong E[d-1]$. Since $1-d\ne 0$ and $1-d\ne d$,
one has $$\Hom_{\mathcal T}(\tau E, E)\cong
\Hom_{\mathcal T}(E[d-1], E)\cong \Hom_{\mathcal T}(E, E[1-d]) = 0$$ where the last equality follows from the definition of an exceptional $1$-cycle. Since $0\ne g_1f_1\in \Hom_{\mathcal T}(\tau E, E)
$, this contradicts $\Hom_{\mathcal T}(\tau E, E) = 0$.

This completes the proof. \end{proof}

\subsection{Proof of Theorem \ref{indsphericalmouth}}  $(1)$ \ This follows from Lemmas  \ref{indspherical} and \ref{sphericalmouth}.

\vskip5pt

$(2)$ \ Let $(E_1, \cdots, E_n)$ be an exceptional cycle in $K^b(A\mbox{-}{\rm proj})$ with $n\ge 3$. Since $\mathcal D^b(k)$ has no exceptional $n$-cycles with $n\ge 3$, one may assume that $\mathcal T\ncong \mathcal D^b(k).$ It suffices to show
that $E_1$ is at the mouth. Otherwise,  the middle term of the Auslander-Reiten triangle starting at $E_1$ is either $0$, or of the form
$Z_1\oplus Z_2$ with $Z_1$ indecomposable and $Z_2\ne 0$. However, the first case is impossible, by Lemma \ref{rv}. Thus there is an Auslander-Reiten triangle
$$E_1\xrightarrow{\binom{f_1}{f_2}} Z_1\oplus Z_2\xrightarrow{(g_1,-g_2)} \tau^{-1}E_1\longrightarrow E_1[1]$$ with $Z_1$ indecomposable and $Z_2\ne 0$.
By $(\rm E2)$, $S(E_n) = E_1[-t]$ for some integer $t$, and hence $\tau^{-1}E_1 \cong S^{-1}(E_1)[-1] = E_n[t-1]$. By Lemma \ref{gfnot0}$(1)$, one has $0\ne g_1f_1: E_1\rightarrow \tau^{-1}E_1$.
Thus  $\Hom(E_1, E_n[t-1])\ne 0$, so $\Hom^{\bullet}(E_1, E_n)\ne 0$. Since $n\ge 3$,  this contradicts the condition $(\rm E3)$. This completes the proof of $(2)$.
\hfill $\square$

 \subsection{} For later applications of Theorem \ref{indsphericalmouth}, 
 we include the following well-known fact.

\begin{Lem}\label{indec} Let $\Lambda$ be a finite-dimensional algebra. Then
$\Lambda$ is indecomposable as an algebra if and only if $K^b(\Lambda\mbox{-}{\rm proj})$ is indecomposable as a triangulated category.

If in addition $\Lambda$ is basic, then
$K^b(\Lambda\mbox{-}{\rm proj})\cong \mathcal D^b(k)$ if and only if $\Lambda = k$. \end{Lem}

\begin{proof} For convenience, we include a proof of the ``only if" parts. Assume that $K^b(\Lambda\mbox{-}{\rm proj})$ $\cong \mathcal T_1\times \mathcal T_2$ with $\mathcal T_1\ne 0$ and $\mathcal T_2\ne 0.$ Let $P = (P_1, P_2)$ be an indecomposable projective $\Lambda$-module with $P_i\in \mathcal T_i$. Since
$\End_\Lambda(P) = \End_{K^b(\Lambda\mbox{-}{\rm proj})}(P) = \End_{\mathcal T_1}(P_1)\times \End_{\mathcal T_2}(P_2)$ is a local algebra, it follows that $\End_\Lambda(P)$ is an indecomposable algebra. Thus
$P\in\mathcal T_1$ or $P\in\mathcal T_2$. Since all the indecomposable projective modules generate $K^b(\Lambda\mbox{-}{\rm proj})$, it follows that both $\mathcal T_1$ and $\mathcal T_2$ contains
at least one indecomposable projective module. Thus $A$ is not indecomposable.

\vskip5pt

Suppose that $\Lambda$ is basic. If $K^b(\Lambda\mbox{-}{\rm proj})\cong \mathcal D^b(k)$, then $\Lambda$ is indecomposable. Since in $\mathcal D^b(k)$ there are no nonzero morphism between two non-isomorphic indecomposable objects, $\Lambda$ has only one isoclass of indecomposable projective module $P$ with $\End(P) = k$. Since $\Lambda$ is basic, $\Lambda = P$ and $\Lambda \cong \End(P) = k$. \end{proof}

\section{\bf Characteristic components of $K^b(A\mbox{-}{\rm proj})$}

Let $A$ be an indecomposable finite-dimensional gentle $k$-algebra.

\subsection{The shape of a characteristic component}

Recall that a connected component of the Auslander-Reiten quiver of $K^b(A\mbox{-}{\rm proj})$ is
{\it a characteristic component}, if it contains a string complex at the mouth.
To get the shape of a characteristic component, we need the following result due to S. Scherotzke.

\begin{Lem}\label{for2cycles} \ {\rm ([Sch, Theorem 4.14, Corollary 3.4]; [V, Theorem])} \  Let $\Lambda$ be a finite-dimensional algebra. If
the Auslander-Reiten quiver of $\mathcal D^b(\Lambda)$ has a component $\Bbb Z\vec\Delta/G$,
where $\Delta$ a Dynkin graph, and $G$ is an admissible automorphism group of $\mathbb{Z}\vec\bigtriangleup$, then
$\mathcal D^b(\Lambda)$ is of finite type $($i.e., it has only finitely many isoclasses of indecomposable objects, up to shift$)$, and $K^b(\Lambda\mbox{-}{\rm proj})=\mathcal D^b(\Lambda)\cong \mathcal D^b(k\vec\bigtriangleup)$.
\end{Lem}

Note that in Lemma \ref{for2cycles}, $\mathcal D^b(\Lambda)$ is not assumed to have Auslander-Reiten triangles (or equivalently, the global dimension of $\Lambda$ is finite. See [H3]).

\begin{Prop}\label{formofcomponent} Let $A$ be an indecomposable finite-dimensional gentle algebra. Then
a characteristic component of $K^b(A\mbox{-}{\rm proj})$ is one of the following$:$
$$\Bbb Z\vec{A_n} \ (n\ge 2), \ \ \ \ \Bbb Z\vec{A_\infty}, \ \ \ \ \Bbb Z\vec{A_\infty}/\langle \tau^n\rangle \ (n\ge 1).$$
\end{Prop}
\begin{proof} Let $C$ be a characteristic component of $K^b(A\mbox{-}{\rm proj})$. Then $C$ contains no loops, by J. Xiao and B. Zhu [XZ, Corollary 2.2.3] (i.e., for a $k$-linear Hom-finite indecomposable triangulated category $\mathcal T$ with Serre functor,
if the Auslander-Reiten quiver of $\mathcal T$ contains a loop, then $X\cong X[1]$ for any object $X\in\mathcal T$).
One may assume that
$C$ has no multiple arrows (otherwise, regarding $C$ as a valued quiver so that it has no multiple arrows). Thus, $C$ is a valued stable translation quiver without loops and multiple arrows.
By Theorem \ref{2terms}, $\alpha(x)\le 2$ for each vertex $x\in C$, where $\alpha(x)$ is
the number of indecomposable direct summands of the middle term in the Auslander-Reiten triangle starting from $x$.
By M. C. R. Butler and C. M. Ringel [BR, p.154] (see also [Rm] and [HPR]),
one has $C=\mathbb{Z}\vec\bigtriangleup/G$, where the underlying graph $\bigtriangleup$ is of type
$$\xymatrix@R=0.2cm@C=0.6cm{A_n \ (n \ge 2): & 1\ar@{-}[r] & 2 \ar@{-}[r] & \cdots \ar@{-}[r] & n &&&
A_{\infty}: & \circ \ar@{-}[r] & \circ \ar@{-}[r]& \circ \ar@{-}[r] & \cdots }$$
$$\xymatrix@R=0.2cm@C=0.6cm{\widetilde{A_{1, 2}}: & \circ \ar@{-}[r]^-{(2, 2)} & \circ &&&&& & A_{\infty}^\infty:
& \cdots \ar@{-}[r] & \circ \ar@{-}[r] & \circ \ar@{-}[r]& \circ \ar@{-}[r] & \cdots}$$
and $G$ is an admissible automorphism group of $\mathbb{Z}\vec\bigtriangleup$. (Since $k$ is algebraically closed, the cases of $C_n, \widetilde{A}_{11}, \widetilde{C}_n, C_\infty$ can not occur).

\vskip5pt

If $\bigtriangleup= A_n$, then $K^b(A\mbox{-}{\rm proj})$ has an Auslander-Reiten component
$C=\mathbb{Z}\vec{A_n}/G$. By the definition of a Auslander-Reiten triangle, one easily check that
an Auslander-Reiten component of $K^b(A\mbox{-}{\rm proj})$ is an
Auslander-Reiten component of $\mathcal D^b(A) = K^{-, b}(A\mbox{-}{\rm proj}).$
Then by  Lemma \ref{for2cycles},  $K^b(A\mbox{-}{\rm proj}) = \mathcal D^b(A)\cong \mathcal D^b(k\vec{A_n})$ and $G = \{1\}$.

\vskip5pt

If $\bigtriangleup=A_{\infty}$, then $G=\{1\}$ or $G=\langle\tau^m\rangle$ for some $m\geq 1$. See [BR, p. 154].

\vskip5pt

Note that $\bigtriangleup$ can not be $\widetilde{A_{1, 2}}$ or $A_{\infty}^\infty$: otherwise $C$ contains no objects at the mouth.

\vskip5pt

This completes the proof. \end{proof}

\subsection{Number of characteristic components} Let $A = kQ/I$ be a finite-dimensional gentle algebra with $Q$ a finite  connected quiver. Since $Q$ is a finite quiver, by definition there are only finitely many forbidden paths, and hence only finitely many forbidden threads.
By Lemma \ref{mouth}, any string complex at the mouth is of the form $P_{m, w}$, where $w$ is a forbidden thread and $m$ is an integer.
Since $P_{m', w}$ is a shift of $P_{m, w}$ for different integers $m$ and $m'$, up to shift, there are only finitely many string complexes at
the mouths.  This shows

\vskip5pt

\begin{Lem} \label{finitecharcomponent} Let $A$ be an indecomposable finite-dimensional gentle algebra. Then there are only finitely many string complexes at
the mouth, and $K^b(A\mbox{-}{\rm proj})$ has only finitely many characteristic components, up to shift.\end{Lem}

\subsection {The AG-invariant of a characteristic component} Let $C$ be a characteristic component of $K^b(A\mbox{-}{\rm proj})$.
For a string complex $X$ at the mouth of $C$, $\tau^l X$ is again a string complex at the mouth for $l\in\Bbb Z$.
Since up to shift there are only finitely many string complexes at
the mouth, there exists a minimal positive integers $n_X$ such that $$\tau^{n_X}X\cong X[m_X-n_X]$$ for some integer
$m_X$. Since $K^b(A\mbox{-}{\rm proj})$ is homotopy-like, such an integer $m_X$ is unique.
We will prove $(n_X, m_X) = (n_Y, m_Y)$ for any  two objects $X$ and $Y$ at the mouth of $C$. For this we need

\begin{fact} \label{an} Let $A$ be an indecomposable finite-dimensional gentle algebra.
Assume that $K^b(A\mbox{-}{\rm proj})$ has a characteristic component $C=\Bbb Z\vec{A_n}$ with $n\ge 2$.
Then $K^b(A\mbox{-}{\rm proj})= \mathcal D^b(A)\cong \mathcal D^b(k\vec{A_n})$,
and for any objects $X$ and $Y$ at the mouth of $C$,
one has $Y = \tau^t X[s]$ for some integers $t$ and $s$.\end{fact}

\begin{proof} Note that $C$ is also an Auslander-Reiten component of $D^b(A)$.
By Lemma \ref{for2cycles}, $K^b(A\mbox{-}{\rm proj})= \mathcal D^b(A)\cong \mathcal D^b(k\vec{A_n})$.

\vskip5pt

Note that the set of indecomposable objects at the mouth of $C$ is the disjoint union of two $\tau$-orbits of $C$, here $\tau$ is the Auslander-Reiten of $D^b(A)$.
Labelling $\vec{A_n}$ as $1\rightarrow 2 \rightarrow \cdots \rightarrow n$,
the indecomposable projective $k\vec{A_n}$-module $P(n)$ and the indecomposable injective $k\vec{A_n}$-module $I(1)$ are at the lower $\tau$-orbit
with $\tau^{n-1}I(1) = P(n)$;
and the indecomposable injective $k\vec{A_n}$-module $I(n)=P(1)$ is at the upper $\tau$-orbit. Then
$\tau P(n) = S(P(n))[-1] = I(n)[-1]$ and $\tau P(1) = S(P(1))[-1] = I(1)[-1] = \tau^{-(n-1)}P(n)[-1].$
Thus $Y = \tau^t X[s]$ for some integers $t$ and $s$. \end{proof}

\begin{Prop} \label{ainfinte} Let $A$ be an indecomposable finite-dimensional gentle algebra, and $C$ a characteristic component of $K^b(A\mbox{-}{\rm proj})$.
Then $(n_X, m_X) = (n_Y, m_Y)$ for any objects $X$ and $Y$ at the mouth of $C$.\end{Prop}
\begin{proof} By Proposition \ref{formofcomponent}, $C$ is either $\Bbb Z\vec{A_n}$ with $n\ge 2$, or $\Bbb Z\vec{A_{\infty}}$, or $\Bbb Z\vec{A_{\infty}}/\langle \tau^n\rangle$ with $n\ge  1$.
If $C=\Bbb Z\vec{A_n}$ with $n\ge 2$, then $Y = \tau^t X[s]$ for some integers $t$ and $s$, by Lemma \ref{an}$(2)$. Thus
$$\tau^{n_X}Y = \tau^t \tau^{n_X} X[s] \cong \tau^t X[m_X-n_X][s] = Y[m_X-n_X].$$
Then $n_Y\le n_X$, by the minimality of $n_Y$. Similarly, $n_X\le n_Y$. Thus $n_X=n_Y$, and then $m_X = m_Y$, by the uniqueness of $m_Y$. If $C$ is either $\Bbb Z\vec{A_{\infty}}$ or $\Bbb Z\vec{A_{\infty}}/\langle \tau^n\rangle$ with $n\ge  1$, then $Y = \tau^t X$ for some integer $t$. By the same argument
one has $(n_X, m_X) = (n_Y, m_Y)$.
\end{proof}

By Proposition \ref{ainfinte}, for each characteristic component $C$ of $K^b(A\mbox{-}{\rm proj})$, there exists a unique pair $(n, m)$ of integers, so that $n$ is the minimal positive integer such that
$\tau^{n}X\cong X[m-n]$, for any indecomposable object $X$ at the mouth of $C$.
We will call $(n, m)$ the {\rm AG}-invariant of $C$. 

\subsection{The height function of a characteristic component} According to Proposition \ref{formofcomponent}, for each characteristic component $C$ of $K^b(A\mbox{-}{\rm proj})$,  there is a well-defined height function $h: C \rightarrow \Bbb N$ such that $h(X) = 1$ if and only if $X$ is at the mouth.
Thus in Theorem \ref{2terms}$(1)$, if $h(P_{m,w}) = n$ then $h(P_{m+m'(w), w_{+}}) = n-1$, $h(P_{m, w^{+}}) = n+1$, and $h(P_{m+m''(w), w^{+}_{+}}) = n$.

\section{\bf Proof of Theorem \ref{morphismbetweenforbiddenthreads} and Corollary \ref{orthogonal}}

A key observation for proving Theorem \ref{morphismbetweenforbiddenthreads} is Lemma \ref{vw} below.
One of the tools in the proof of Lemma \ref{vw} is a description of
morphisms between string complexes in the bounded complex category $C^b(A\mbox{-}{\rm proj})$ of $A\mbox{-}{\rm proj}$,
via single maps, double maps, and graph maps, introduced in \cite{ALP}.

\subsection{Morphisms between some indecomposable objects of $K^b(A\mbox{-}{\rm proj})$}
Let $A$ be a gentle algebra.
We use {\bf the  conventions:} We denote by $P_w$ the string complex $P_{m, w}$ (if we do not need to specify $m$); and if $\mu$ is an automorphism of a $1$-dimensional vector space, we also denote by $P_w$ the band complex $P_{m,w, \mu}$ (if no confusions caused. In fact, to study exceptional cycles, we only need to consider this special case of $\mu$).
In this way some results can be stated for both string complexes and band complexes for $\mu\in k^*$.

\vskip5pt

We use the unfolded diagram in \cite{ALP} to express $P_w$.
This presentation is based on properties of gentle algebras. We write a right multiplication $r_p$ by a path $p$ simply as $p$.
Thus, if $f = r_p$ and $g = r_q$ then the composition $gf = r_{pq}$ (sometimes $gf$ is written as $pq$. In this way $gf$ is written $fg$ below). In the following $v$ and $w$ are homotopy strings or homotopy bands.

\vskip5pt

{\bf Single maps.} \ Suppose that we are given a chain map between unfolded diagrams
$$\xymatrix@R=0.5cm{
P_v:    & \cdots  \bullet \ar@{-}[r]^{v_L}  &   \bullet \ar@{-}[r]^{v_R} \ar[d]^f  &  \bullet  \cdots &\\
P_w:    & \cdots \bullet \ar@{-}[r]^{w_L}  &   \bullet \ar@{-}[r]^{w_R}   &  \bullet  \cdots &
}$$
with permitted path $f$. A chain map in $\Hom_{C^b(A\mbox{-}{\rm proj})}(P_v,P_w)$ is {\it a single map} if it has only one non-zero component $f$, and
$f$ satisfies the following conditions ([ALP, 3.1]):

\text{\rm (L)} \ If $v_L$ is direct, then $v_Lf=0$; and if $w_L$ is inverse, then $fw_L=0$.

\text{\rm (R)} \ If $v_R$ is inverse, then $v_Rf=0$; and if $w_R$ is direct, then $fw_R=0.$

\noindent Denote the set of single maps $P_v\rightarrow P_w$ by $S_{v,w}$.

\vskip5pt

{\bf Double maps.} \ Suppose that we are given a chain map between unfolded diagrams:
$$\xymatrix@R=0.5cm{
P_v:    & \cdots   \bullet \ar@{-}[r]^{v_L}  &  \bullet \ar[r]^{v_M} \ar[d]_{f_L} & \bullet \ar@{-}[r]^{v_R} \ar[d]^{f_R}  &  \bullet  \cdots &\\
P_w:    & \cdots   \bullet \ar@{-}[r]^{w_L}  &   \bullet \ar[r]^{w_M}  &\bullet \ar@{-}[r]^{w_R}   &  \bullet  \cdots &
}$$
with permitted paths $f_L$ and $f_R$, such that $v_Mf_R=f_Lw_M\neq0$. A map in $\Hom_{C^b(A\mbox{-}{\rm proj})}(P_v,P_w)$ is {\it a double map} if it has only two consecutive non-zero components
$f_L$ and $f_R$, such that $f_L$ satisfies \text{\rm (L)} and $f_R$ satisfies \text{\rm (R)}. See [ALP, 3.3]. Write $D_{v,w}$ for the set of double maps $P_v\rightarrow P_w$.

\vskip5pt

{\bf Graph maps.} \ Suppose that we are given a chain map between unfolded diagrams
$$\xymatrix@R=0.6cm{
P_v:    & \cdots  \bullet \ar@{-}[r]^{v_L} \ar@<1.5ex>@{-->}[d]_{f_L} &  \bullet \ar@{-}[r]^{u_p} \ar@{=}[d]^{\rm Id} &
\bullet \ar@{-}[r]^{u_{p-1}} \ar@{=}[d]^{\rm Id} &\cdots  \ar@{-}[r]^{u_2} &\bullet \ar@{-}[r]^{u_1} \ar@{=}[d]^{\rm Id} & \bullet \ar@{-}[r]^{v_R}\ar@{=}[d]^{\rm Id} &  \bullet \cdots\ar@<-1.5ex>@{-->}[d]^{f_R}  &\\
P_w:    & \cdots  \bullet \ar@{-}[r]^{w_L}  &  \bullet \ar@{-}[r]^{u_p}  &
\bullet \ar@{-}[r]^{u_{p-1}}  &\cdots  \ar@{-}[r]^{u_2} &\bullet \ar@{-}[r]^{u_1}  & \bullet \ar@{-}[r]^{w_R}   &  \bullet \cdots &\\
}$$
where $v_L\neq w_L$ and $v_R\neq w_R$, $f_L$ is either a permitted path or zero,  and $f_R$ is either a permitted path or zero.
A map in $\Hom_{C^b(A\mbox{-}{\rm proj})}(P_v,P_w)$ is {\it a graph map} if one of the following conditions \text{\rm (LG1)} and \text{\rm (LG2)} holds, and one of \text{\rm (RG1)} and \text{\rm (RG2)} holds, where

\vskip5pt

\noindent\text{\rm (LG1)} \ The arrows $v_L$ and $w_L$ are either both direct or both inverse. In this case, there exists some (scalar multiple of a) permitted path $f_L$ such that the square on the left commutes.\\
\text{\rm (LG2)} \ The arrows $v_L$ and $w_L$ are neither both direct nor both inverse. In this case, if $v_L$ is non-zero then it is inverse, and if $w_L$ is non-zero then it is direct. \\
\text{\rm (RG1)} \ The arrows $v_R$ and $w_R$ are either both direct or both inverse. In this case, there exists some (scalar multiple of a) permitted path $f_R$ such that the square on the right commutes.\\
\text{\rm (RG2)} \ The arrows $v_R$ and $w_R$ are neither both direct nor both inverse. In this case, if $w_R$ is non-zero then it is inverse, and if $v_R$ is non-zero then it is direct.

Denote the set of graph maps $P_v\rightarrow P_w$ by $G_{v,w}$.

\vskip5pt

\begin{Lem} \label{alp} $($\cite{ALP}$, \ {\rm 4.1})$ \ Let $A$ be a gentle algebra. The set $B_{v,w}:=S_{v,w} \cup D_{v,w} \cup G_{v,w}$ is a basis of $\Hom_{C^b(A\mbox{-}{\rm proj})}(P_v,P_w)$.

\vskip5pt

Moreover, if both $v$  and $w$ are forbidden threads, then $G_{v,w}$ contains at most one graph map.
\end{Lem}
\begin{proof} We only need to justify the second assertion. Assume that both $v$ and $w$ are forbidden threads, and $f\in G_{v,w},$ i.e.,
$f$ is a graph map from $P_v$ to $P_w$.

\vskip5pt

{\bf Case 1.} \ If $f$ has only one non-zero component, then $f$ is of the form:
$$\xymatrix@R=0.6cm{
P_v:  & \cdots \bullet\ar[r]^{v_L}& \bullet\ar@{=}[d] \ar[r]^{v_R} &\bullet \cdots  \\
P_w:  & \cdots \bullet\ar[r]^{w_L}& \bullet \ar[r]^{w_R} &\bullet \cdots
}$$
Since both $v$ and $w$ are forbidden threads, only $(\rm LG2)$ and $(\rm RG2)$ is possible. Thus we have $v_L=0=w_R$, and hence
$f=e(v)=s(w)$. Since $v$ and $w$ are given, it follows that such an $f$ is unique.

\vskip5pt

{\bf Case 2.} \ If $f$ has at least two non-zero component,  then $f$ is of the form:
$$\xymatrix@R=0.6cm{
P_v:  & \cdots \bullet\ar[r]^{v_L}& \bullet\ar@{=}[d]\ar[r]^{v_M} & \bullet \ar@{=}[d] \ar[r]^{v_R} &\bullet \cdots  \\
P_w:  & \cdots \bullet\ar[r]^{v_L}& \bullet\ar[r]^{v_M} & \bullet \ar[r]^{v_R} &\bullet \cdots
}$$
Since any two different forbidden threads has no common arrows, it follows that $v=w$ and $f$ is identity.

\vskip5pt

This completes the proof.
\end{proof}

\subsection{} We are now in position to state a main lemma for proving Theorem \ref{morphismbetweenforbiddenthreads}.

\begin{Lem}\label{vw}
Let $v$ and $c$ be forbidden threads,  $m$ and $\tilde {m}$ integers, $f: P_{m,v}\rightarrow P_{\tilde {m}, c}$  a non-zero non-isomorphism in $K^b(A\mbox{-}{\rm proj})$. Then

\vskip5pt

$(1)$ \ The chain map $f$ is of the form:
$$\xymatrix@R=0.5cm{
P_{m, v}:    & &  \bullet \ar[r] \ar[d]_{f=\lambda h}  &  \bullet  \cdots  \\
P_{\tilde {m}, c}:   & \cdots\bullet \ar[r]& \bullet     &
}$$
where $\lambda\in k^{*},$ and $h$ is a permitted thread satisfying
\begin{align*} e(h)=e(v), \ \ \ & e'(h)=-e'(v)   \\ s(h)=s(c),\ \ \ & s'(h)=-s'(c).\end{align*}
$($Thus, the unique nonzero component of $f$ is from the first nonzero component of $P_{m, v}$ on the left hand to the the first nonzero component of $P_{\tilde {m}, c}$ on the right hand.$)$
\vskip5pt

$(2)$ \ If $g: P_{m,v}\rightarrow P_{m', w}$ is also a non-zero non-isomorphism in $K^b(A\mbox{-}{\rm proj})$, where $w$ is a forbidden thread and $m'$ an integer,
then we have $w=c$, $P_{m',w}\cong P_{\tilde{m},c}$ and $g\in kf$. \end{Lem}

\begin{proof}
$(1)$ \ By [ALP, Proposition 4.1] (c.f. Lemma \ref{alp}), $f$ is a linear combination of single maps, double maps and graph maps.
Let $f_i\ (1\leq i \leq n)$ be a basis of $\Hom_{C^b(A\mbox{-}{\rm proj})}(P_{m, v}, P_{\tilde {m}, c})$, where $f_1$ is a graph map, $f_i\ (i\in I_1)$ are single maps, $f_i\ (i\in I_2)$ are double maps,
and $\{1, \cdots, n\}=\{1\}\cup I_1\cup I_2$. Thus in
$\Hom_{C^b(A\mbox{-}{\rm proj})}(P_{m, v}, P_{\tilde {m}, c})$, $f=\sum^{n}_{i=1}\lambda_if_i$ for each $\lambda_i\in k$. But in $\Hom_{K^b(A\mbox{-}{\rm proj})}(P_{m, v}, P_{\tilde {m}, c})$, some $f_i$
possibly become zero. We analyse all the $f_i$ such that $f_i\ne 0$ in $K^b(A\mbox{-}{\rm proj})$. In the following, we denote $f_i$ by $f'$.

\vskip5pt

{\bf Step 1.} \  Suppose that $f'$ is a single map:
$$\xymatrix@R=0.6cm{
P_{m, v}:    & \cdots \bullet \ar[r]^{v_L}  &   \bullet \ar[r]^{v_R} \ar[d]_{f'}  &  \bullet  \cdots&\\
P_{\tilde {m}, c}:    & \cdots \bullet \ar[r]^{c_L}  &   \bullet \ar[r]^{c_R}   &  \bullet  \cdots&
}$$
Thus  $f'$ is a permitted path. There are $4$ cases:

{\rm (i)}:   both of $v$ and $c$ are trivial forbidden threads;

{\rm (ii)}:   $v$ is a trivial forbidden thread, and $c$ is a non-trivial forbidden thread;

{\rm (iii)}:  $v$ is a non-trivial forbidden thread, and $c$ is a trivial forbidden thread;

{\rm (iv)}:   both of $v$ and $c$ are non-trivial forbidden threads.

\noindent We prove that $(1)$ holds in all these $4$ cases.

\vskip5pt

If $v$ is a trivial forbidden thread, then $e'(f')=-e'(v), \ e(f')=e(v)$. And if there exists an arrow $\beta$ with $s(\beta)=e(f')$, then $\beta f'=0$  (otherwise, $\beta f'\ne 0$. This contradicts the assumption that $v$ is a trivial forbidden thread). Similarly, if $c$ is a trivial forbidden thread, then $s'(f')=-s'(c)$, \ $s(f')=s(c)$. And if there exists an arrow $\gamma$ with $e(\gamma)=s(f')$, then $f'\gamma=0$.

\vskip5pt

If $v$ is a non-trivial forbidden thread, then $v_R \ne 0$ (otherwise $v_R = 0$.  By ${\rm(L)}$, $v_Lf'=0$.  This contradicts the assumption that $v$ is a forbidden thread).
{\bf Claim 1:}  $e'(f')=-e'(v_R)$. Otherwise we have $e'(f')=e'(v_R)$. Since $v_R$ is an arrow, $v_R$ is a subletter of $f'$. Thus either $f'=v_R$, or $f'=v_R\alpha$, where $\alpha$ is a permitted path.
If $f'=v_R$, then $f'$ is null homotopic as indicated below, which contradicts $f'\ne 0$:
$$\xymatrix@R=0.6cm{
P_{m, v}:& \cdots \bullet  \ar[r]^{v_L} & \bullet\ar@{-->}[ld]_{0} \ar[d]_{f'} \ar[r]^{v_R} & \bullet\ar@{--}[r]\ar@{-->}[ld]_{1} \ar[r]^{v_1} &\bullet \ar@{-->}[ld]_{-1}&\cdots & \bullet \ar@{-->}[ld]\ar[r]^{v_n}&\bullet\ar@{-->}[ld]_-{(-1)^n}\\
P_{\tilde {m}, c}:& \cdots \bullet   \ar[r]^{c_L} & \bullet \ar[r]^{c_R} & \bullet &\cdots& \bullet  \ar[r]^{c_{n-1}}& \bullet&&}$$
If $f'=v_R\alpha$, then $f'$ is again null homotopic:

$$\xymatrix{
P_{m, v}:    & \cdots \bullet \ar[r]^{v_L}  &   \bullet \ar[r]^{v_R} \ar[d]_{f'} \ar@{-->}[ld]_0 &  \bullet  \cdots\ar@{-->}[ld]_{\alpha}&\\
P_{\tilde {m}, c}:    & \cdots \bullet \ar[r]^{c_L}  &   \bullet \ar[r]^{c_R}   &  \bullet  \cdots&
}$$
This proves {\bf Claim 1}, i.e., $e'(f')=-e'(v_R)$. It follows that $v_L=0$ (otherwise $v_L\ne 0$. By ${\rm (L)}$, $v_Lf'=0=v_Lv_R$. This contradicts $e'(f')=-e'(v_R)$). Hence $e'(f')=-e'(v_R)=-e'(v)$.
{\bf Claim 2}:  If there exists an arrow $\beta$ with $s(\beta)=e(f')$, then $\beta f'=0$ (otherwise, $\beta f'\ne 0$. According to ${\rm (G2)}$ and ${\rm (G3)}$ in the definition of a gentle algebra, we have $\beta v=0$. This contradicts  the assumption that $v$ is a forbidden thread).

\vskip5pt

In conclusion, if $v$ is a non-trivial forbidden thread, then $v_R \ne 0, \ e'(f')=-e'(v), \ e(f')=e(v)$. And if there exists an arrow $\beta$ with $s(\beta)=e(f')$, then $\beta f'=0$ (otherwise, $\beta f'\ne 0$. This contradicts the assumption that $v$ is a forbidden thread). Similarly, if $c$ is a non-trivial forbidden thread, then $c_R = 0, \ s'(f')=-s'(c)$, \ $s(f')=s(c)$. And if there exists an arrow $\gamma$ with $e(\gamma)=s(f')$ then $f'\gamma=0$.

\vskip5pt

Putting together, in all the case {\rm (i)} - {\rm (iv)}, any non-zero single map $f' : P_{m,v}\rightarrow P_{\tilde {m}, c}$ is of form:
$$\xymatrix@R=0.5cm{
P_{m, v}:    & &  \bullet \ar[r] \ar[d]_{f'}  &  \bullet  \cdots  \\
P_{\tilde {m}, c}:   & \cdots\bullet \ar[r]& \bullet     &
}$$
and $f'$ is a (non-trivial) permitted thread with \begin{align*} e(f')=e(v), \ \ \ & e'(f')=-e'(v)   \\ s(f')=s(c),\ \ \ & s'(f')=-s'(c).\end{align*}

\vskip5pt

{\bf Step 2.} \  Suppose that $f'$ is a double map:
$$\xymatrix@R=0.6cm{
P_{m, v}:    & \cdots \bullet \ar[r]^{v_L}  &  \bullet \ar[r]^{v_M} \ar[d]_{f'_L} & \bullet \ar[r]^{v_R} \ar[d]_{f_R'} &  \bullet  \cdots&\\
P_{\tilde {m}, c}:    & \cdots \bullet \ar@{-}[r]^{c_L}  &   \bullet \ar[r]^{c_M}  &\bullet \ar@{-}[r]^{c_R}   &  \bullet \cdots&
}$$
where $v_Mf_R'=f_L'c_M\neq0$ for permitted paths $f_L'$ and $f_R'$.
Since  there are no commutative relations in a gentle algebra,  $v_Mf_R'$ and $f_L'c_M$ are the same path, and hence there exists a path $h$ such that $f_L'=v_Mh$ and $f_R'=hc_M$.
Hence $f'$ is null homotopic:
$$\xymatrix@R=0.6cm{
P_{m, v}:    & \cdots \bullet \ar[r]^{v_L}  &  \bullet \ar@{-->}[ld]_0\ar[r]^{v_M} \ar[d]_{f_L'} & \bullet \ar[r]^{v_R} \ar[d]_{f_R'} \ar@{-->}[dl]_h  &  \bullet \ar@{-->}[ld]_0 \cdots&\\
P_{\tilde {m}, c}:    & \cdots \bullet \ar[r]^{c_L}  &   \bullet \ar[r]^{c_M}  &\bullet \ar[r]^{c_R}   &  \bullet\cdots&
}$$
This contradiction  means that a double map can not appear as $f'$.

\vskip5pt

{\bf Step 3.} \ Suppose that $f'$ is a graph map, then $f'$ is either the identity graph map or a non-identity graph map. Here we only consider the case that $f'$ is a non-identity graph map, since in
{\bf Step 4.} we will prove that $f'$ can not be an identity graph map.

\vskip5pt

By the definition, two different forbidden threads have no common arrows. It follows from ${\rm (LG2)}$ and ${\rm (RG2)}$ that a non-zero non-identity graph map $f'$ is as follows:
$$\xymatrix@R=0.5cm{
P_{m, v}:    & &  \bullet \ar[r] \ar@{=}[d]_{f'}  &  \bullet  \cdots  \\
P_{\tilde {m}, c}:   & \cdots\bullet \ar[r]& \bullet     &
}$$
We need to prove that $f'$ is a trivial permitted thread, and then the relations automatically hold: \begin{align*} e(f')=e(v), \ \ \ & e'(f')=-e'(v)   \\ s(f')=s(c),\ \ \ & s'(f')=-s'(c).\end{align*}

\vskip5pt

There are $4$ cases:

{\rm (i)}:   both of $v$ and $c$ are trivial forbidden threads;

{\rm (ii)}:   $v$ is a trivial forbidden thread, and $c$ is a non-trivial forbidden thread;

{\rm (iii)}:  $v$ is a non-trivial forbidden thread, and $c$ is a trivial forbidden thread;

{\rm (iv)}:   both of $v$ and $c$ are non-trivial forbidden threads.

\vskip5pt

The case {\rm(i)} can not occur, since $f'$ is a non-identity graph map.

In the case {\rm (ii)}, since $v$ is a trivial forbidden thread, it has  the following possibilities:\\
$(1)$\ there is exactly one arrow $\beta$ with $e(\beta)=v$ and no arrows starting at $v$;  \\
$(2)$\ there is exactly one arrow $\gamma$ with $s(\gamma)=v$ and no arrows ending at $v$;  \\
$(3)$\ there is exactly one arrow $\beta$ with $e(\beta)=v$, and exactly one arrow $\gamma$ with $s(\gamma)=v$, such that $\gamma\beta=0$.
Since $c$ is a non-trivial forbidden thread, the situation $(1)$ can not occur. Also, the situation $(3)$ can not occur
(otherwise, $c\beta=0$, this contradicts $c$ is a non-trivial forbidden thread).
Hence there is exactly one arrow $\gamma$ with $s(\gamma)=v$ and no arrows ending at $v$. In this case, $f'=v$ is a trivial permitted thread.

In the case {\rm (iii)}, we can similarly prove that $f'$ is a trivial permitted thread.

In the case {\rm (iv)}, since both of $v$ and $c$ are non-trivial forbidden threads, there exist  exactly one arrow $\beta$ with $e(v)=e(\beta)=f'$ and exactly one arrow $\gamma$ with $s(c)=s(\gamma)=f'$ such that $\gamma\beta\ne 0$.
Thus $f'$ is a trivial permitted thread.

In conclusion, any non-identity graph map $f'$ is of form:
$$\xymatrix@R=0.5cm{
P_{m, v}:    & &  \bullet \ar[r] \ar@{=}[d]_{f'}  &  \bullet  \cdots  \\
P_{\tilde {m}, c}:   & \cdots\bullet \ar[r]& \bullet     &
}$$
and $f'$ is a trivial permitted thread.

\vskip5pt

{\bf Step 4.} \ By the assumption, $f$ is a non-zero non-isomorphism from $P_{m, v}$ to $P_{\tilde {m}, c}$. Write $f$ as
$f=\lambda_1f_1+\sum\limits_{2\le i\le m} \lambda_if_i$ in $K^b(A\mbox{-}{\rm proj})$, where all $\lambda_i\in k$,  $f_1$ is a graph map, $f_i \ (2\le i\le m)$ are single maps
(note that by {\bf Step 2}, double maps can not appear).

\vskip5pt

$(\rm i)$ \ \ If $f_1$ is a non-identity graph map, by {\bf Step 1} and {\bf Step 3}, then each $f_i$ is a permitted thread and $e'(f_i)=-e'(v)$ and $e(f_i)=e(v)$.
That is, $\Phi_1: \mathcal{P}_A\rightarrow \mathcal{F}_A,\ f_i\longmapsto v$ (c.f. Subsection 2.5) for all $i$.
Since $\Phi_1$ is a bijection, it follows that all the $f_i$ are the same.
This is a contradiction, since $f_1$ is a trivial permitted thread and $f_2, \cdots, f_m$ are non-trivial permitted threads. So $f_1$ and $f_2= \cdots =f_m$ can not appear simultaneously as basis elements of $\Hom_{K^b(A\mbox{-}{\rm proj})}(P_{m, v}, P_{\tilde {m}, c})$.
Denoted $f_1$ or $f_2$ by $h$ (thus $h$ is a permitted thread). Hence $f=\lambda h$ with $\lambda\in k^*$, and we are done.

\vskip5pt

$(\rm ii)$ \ Assume that $f_1$ is the identity graph map. As in case $(\rm i)$, all the $f_i$ for $2\le i\le m$ are the same, again denoted by $h$. Thus  $h$ is a permitted thread and
$e'(h)=-e'(v)$ and $e(h)=e(v)$.  So $f=\lambda_1f_1+\lambda h$  where $\lambda\in k$. Since $f_1$ is the identity,
all the morphisms $f, f_1, h$ can be regarded as elements in $\End(P_{m,v})$, such that $f$ and $h$ are non-invertible elements.  Since $\End(P_{m,v})$ is a local algebra, it follows that $\lambda_1f_1$ is also non-invertible element in $\End(P_{m,v})$, it follows that  $\lambda_1=0$ and $f = \lambda h$.

\vskip5pt

This completes the proof of $(1)$.

\vskip5pt

$(2)$ According to $(1)$, $g: P_{m,v}\rightarrow P_{m', w}$ is of the form:

$$\xymatrix@R=0.5cm{
P_{m, v}:    & &  \bullet \ar[r] \ar[d]_{g=\lambda' h'}  &  \bullet  \cdots  \\
P_{m', w}:   & \cdots\bullet \ar[r]& \bullet     &
}$$
where $h'$ is a permitted thread and $\lambda\in k^{*}$, satisfying \begin{align*} e(h')=e(v), \ \ \ & e'(h')=-e'(v)   \\ s(h')=s(w),\ \ \ & s'(h')=-s'(w).\end{align*}
This means that the map $\Phi_1: \mathcal{P}_A\rightarrow \mathcal{F}_A$ sends $h'$ to $v$, and $\Phi_2: \mathcal{F}_A\rightarrow \mathcal{P}_A$ sends $w$ to $h'$.

\vskip5pt

By $(1)$, we have $\Phi_1(h)=v$ and $\Phi_2(c)=h$.
Since $\Phi_1$ is a bijection, it follows that  $h=h'$. While $\Phi_2$ is also a bijection, it follows that $c=w$.
Thus, $P_{m', w}$ is a shift of $P_{\tilde{m}, c}$. But from above two diagrams we see that the first non-zero position from right to left of the two complexes $P_{m', w}$ and $P_{\tilde{m}, c}$ are the same.
Hence $P_{m', w} \cong P_{\tilde{m}, c}$, and $g\in kf$.
\end{proof}

\subsection{Proof of Theorem \ref{morphismbetweenforbiddenthreads}}

\vskip5pt

Let $M$ and $M'$ be string complexes at the mouth. Thus $A\ne k$.

\vskip5pt

{\bf Step 1}. \ Since  $A$ is indecomposable and $A\ne k$, it follows from Lemma \ref{indec} that
$K^b(A\mbox{-}{\rm proj})$ is an indecomposable and $K^b(A\mbox{-}{\rm proj})\ncong \mathcal D^b(k)$. According to Lemma \ref{rv}, there exists a non-zero morphism  $f: M \longrightarrow S(M)$ such that $f$ is not an isomorphism.

\vskip5pt

{\bf Step 2}. \ If $g\in\Hom_{K^b(A\mbox{-}{\rm proj})}(M, S(M))$ is non-zero and non-isomorphism, then $g\in kf$.

\vskip5pt

Since $M$ is a string complex at the mouth, by Lemma \ref{mouth}, $M = P_{m, v}$, where $v$ is a forbidden thread and $m$ is an integer. Similarly, $S(M) = P_{\tilde{m}, c}$ for a forbidden thread $c$ and an integer $\tilde{m}$.

\vskip5pt

Since $f: P_{m, v} \longrightarrow P_{\tilde{m},c} = S(P_{m,v})$ is a non-zero morphism which is not an isomorphism, it follows from
Lemma \ref{vw}$(2)$ that $g\in kf$.

\vskip5pt

{\bf Step 3}. \ If there is an non-zero non-isomorphism $g\in\Hom_{K^b(A\mbox{-}{\rm proj})}(M, M')$, then
$M'\cong S(M)$ and $g\in kf.$

\vskip5pt

In fact, write $M' = P_{m', w}$ for a forbidden thread $w$ and an integer $m'$. Since we have already a non-zero non isomorphism $f: P_{m, v} \longrightarrow P_{\tilde{m},c} = S(M)$,
it follows from Lemma \ref{vw}$(2)$ that $P_{m', w}\cong S(M)$ and $g\in kf.$
\vskip5pt

\vskip5pt

{\bf Step 4}. \ One has \ \  $\dim_k \End(M)= \dim_k\Hom_{K^b(A\mbox{-}{\rm proj})}(M, S(M)) = \begin{cases} 1, & \mbox{if} \ M\ncong S(M); \\ 2,
& \mbox{if} \ M\cong S(M).\end{cases}$

\vskip5pt

In fact, if $M\ncong S(M)$, then any non-zero morphism $g: M\longrightarrow S(M)$ is of course not an isomorphism,  and hence by {\bf Step 2}, $g\in kf$. It follows that $\dim_k\Hom_{K^b(A\mbox{-}{\rm proj})}(M, S(M)) = 1$. Thus $\dim_k \End(M) = 1$.
If $M\cong S(M)$ with an isomorphism $\sigma: S(M)\longrightarrow M$, then $0\ne \sigma f\in {\rm rad}\End(M)$.
By {\bf Step 3}, any non-zero element in ${\rm rad}\End(M)$ is in $kf$, thus $\dim_k {\rm rad}\End(M) = 1.$ While
$\End(M)/{\rm rad}\End(M)\cong k$, one has $\dim_k \End(P_{m, v}) = 2.$

\vskip5pt

{\bf Step 5}. \  One has $$\dim_k\Hom_{K^b(A\mbox{-}{\rm proj})}(M, M')= \begin{cases} 1, & \ \ \ \mbox{if} \ M'\cong S(M) \ \mbox{and} \ M'\ncong M; \\ 2,  & \ \ \ \mbox{if} \ M'\cong S(M)\cong M; \\ 1, & \ \ \
\mbox{if} \ M'\ncong S(M) \ \mbox{and} \ M'\cong M;
\\ 0, & \ \ \ \mbox{if} \ M'\ncong S(M)\ \mbox{and} \ M'\ncong M.\end{cases}$$

In fact, if $M'\cong S(M)$, then the assertion follows from {\bf Step 4}. If $M'\ncong S(M)$ and $M'\cong M$, then the assertion again follows from {\bf Step 4}.
If $M'\ncong S(M)$ and $M'\ncong M$, then any nonzero morphism (if there exists) in $\Hom_{K^b(A\mbox{-}{\rm proj})}(M, M')$ should be an isomorphism, by {\bf Step 3}. But, since
$M'\ncong M$, it follows that $\Hom_{K^b(A\mbox{-}{\rm proj})}(M, M') = 0.$

\vskip5pt

Now, Theorem \ref{morphismbetweenforbiddenthreads} is just a reformulation of {\bf Step 5}. \hfill $\square$

\subsection{Proof of Corollary \ref{orthogonal}.}  \  Let $C$ and $C'$ are different characteristic components of $K^b(A\mbox{-}{\rm proj})$, up to shift, $X\in C$, and $Y\in C'$. Since $Y$ and $Y[i]$ are in the same characteristic component $C'$, up to shift, it suffices to show that $\Hom_{K^b(A\mbox{-}{\rm proj})}(X, Y) = 0$ for each $X\in C$ and each $Y\in C'$.
Use  double induction on the height $h(X)$ and $h(Y)$  (cf., Subsection 4.5). Assume that $h(X) = 1$, i.e., $X$ is at the mouth of $C$. If $h(Y)=1$, i.e., $Y$ is at the mouth of $C'$, then
$Y\ncong S(X)$ and $Y\ncong X$, since $C'$ and $C$ are different characteristic components, up to shift. It follows from Theorem \ref{morphismbetweenforbiddenthreads} that
$\Hom_{K^b(A\mbox{-}{\rm proj})}(X, Y) = 0$. If $h(Y)=2$, then 
by the shape of $C'$ (cf. Proposition \ref{formofcomponent}) there is an Auslander-Reiten triangle $\tau Z\longrightarrow Y\longrightarrow Z\longrightarrow S(Z)$ with $h(Z)=h(\tau Z)=1$. Applying $\Hom_{K^b(A\mbox{-}{\rm proj})}(X,-)$ to this Auslander-Reiten triangle one sees $\Hom_{K^b(A\mbox{-}{\rm proj})}(X, Y) = 0$.

Suppose that the assertion holds for $h(X) = 1$ and $h(Y)=m \ (m\ge 2)$. We prove the assertion holds also for
$h(Y)=m+1\ge 3$. By Proposition \ref{formofcomponent} there is an Auslander-Reiten triangle $\tau Z\longrightarrow Y\oplus Y' \longrightarrow Z\longrightarrow S(Z)$ with $h(Y')=m-1$, $h(Z)=h(\tau Z)=m$. Again applying $\Hom_{K^b(A\mbox{-}{\rm proj})}(X,-)$ and using the inductive hypothesis one gets $\Hom_{K^b(A\mbox{-}{\rm proj})}(X, Y) = 0$.

\vskip5pt

Thus, we have proved the assertion for $h(X) = 1$. By the same argument one can show the assertion for for $h(X) = 2$. Assume that the assertion holds for $h(X) = n \ (n\ge 2)$. By the similar argument one sees that the assertion holds also for
$h(X)=n+1$. This completes the proof. \hfill $\square$

\section{\bf Proof of Theorem \ref{main}}

\subsection{Exceptional $2$-cycles of the form $(E, E)$} We first point out the following fact.

\begin{Prop}\label{2cycles} Let $\mathcal{T}$ be an indecomposable $k$-linear Hom-finite Krull-Schmidt triangulated category with Serre functor. Assume that  $\mathcal{T}$ is homotopy-like.
Then there exists an exceptional $2$-cycle of the form  $(E, E)$ in $\mathcal{T}$ if and only if $\mathcal T\cong \mathcal D^b(k).$
\end{Prop}
\begin{proof} It is clear that $(k, k)$ is an exceptional cycle in $\mathcal D^b(k).$ Assume that $(E, E)$ is an exceptional cycle in $\mathcal{T}$. Then $S(E)\cong E[t]$ for some integer $t$, by ${\rm(E2)}$.
Since $\Hom_\mathcal T(E, E[t]) = \Hom_\mathcal T(E, S(E)) \cong \End(E)\cong k,$ by ${\rm(E1)}$ one has $t=0$. Thus there is an isomorphism $\sigma: E\longrightarrow S(E)$.
We then claim $\mathcal T\cong \mathcal D^b(k).$
Otherwise, by Lemma \ref{rv}, there is a nonzero non-isomorphism $f: E\longrightarrow S(E)$. Since
$\sigma$ and $f$ are linearly independent, $\dim_k\End_\mathcal T(E) = \dim_k\Hom_\mathcal T(E, S(E))\ge 2$. This
contradicts $\End(E)\cong k$.
\end{proof}

\subsection{Exceptional $2$-cycles} Exceptional $2$-cycles in $K^b(A\mbox{-}{\rm proj})$  are more complicated than the other cases. This phenomenon is caused by the fact that there is no restriction of the condition {\rm (E3)}.

\begin{Lem}\label{2cyclefornonA3} \ Let $A = kQ/I$ be a finite-dimensional gentle algebra with $A\ne k$, where $Q$ is a finite connected quiver such that the underlying graph of $Q$ is not of type $A_3$, $(X_1, X_2)$ an exceptional cycle in $K^b(A\mbox{-}{\rm proj})$.
Then $X_1$ is at the mouth of a characteristic component of {\rm AG}-invariant $(2, m)$, and
$(X_1, X_2)\approx (X_1, \tau X_1)$.
\end{Lem}
\begin{proof} By Lemma \ref{indec}, $K^b(A\mbox{-}{\rm proj})$ is indecomposable and $K^b(A\mbox{-}{\rm proj})\ncong \mathcal D^b(k)$.

First, $X_1$ is a string complex. Otherwise,  $X_1$ is a band complex. Then  $\tau X_1 = X_1$ (cf. Theorem \ref{2terms}$(2)$), and hence by $(\rm E1)$ one gets a contradiction
$$0 = \Hom(X_1, X_1[1])=\Hom(X_1, \tau X_1[1]) = \Hom(X_1, S(X_1))\cong \Hom(X_1, X_1)^*.$$

Second, we show that $X_1$ is at the mouth. Otherwise, the middle term of the Auslander-Reiten triangle starting at $X_1$ is either $0$, or of the form
$Z_1\oplus Z_2$ with $Z_1$ and $Z_2$ indecomposable (cf. Theorem \ref{2terms}$(1)$). However, the first case is impossible, by Lemma \ref{rv}.
Thus there is an Auslander-Reiten triangle $$X_1\xrightarrow{\binom{f_1}{f_2}} Z_1\oplus Z_2\xrightarrow{(g_1,-g_2)} \tau^{-1}X_1\longrightarrow X_1[1]$$ with $Z_1$ and $Z_2$ indecomposable.
Since $(X_1, X_2)$ is an exceptional cycle, $\End(X_1)\cong k$. By Lemma \ref{gfnot0}$(1)$ one has $0\ne g_1f_1: X_1\longrightarrow \tau^{-1}X_1$.
According to Lemma \ref{con1}, there is a distinguished triangle $${\rm Cone}(f_1)[-1]\oplus {\rm Cone}(g_1)[-1]\longrightarrow X_1 \xrightarrow{g_1f_1} \tau^{-1}X_1\longrightarrow {\rm Cone}(f_1)\oplus {\rm Cone}(g_1).$$
By Lemma \ref{cone2}, ${\rm Cone}(f_1)$ and ${\rm Cone}(g_1)$ are at the mouth.
By the assumption $\tau X_1$ is also not at the mouth, thus there is an Auslander-Reiten triangle
$$\tau X_1\longrightarrow Y_1\oplus Y_2 \longrightarrow X_1\stackrel h \longrightarrow S(X_1)$$ where
both $Y_1$ and $Y_2$ are indecomposable. We claim that at most one of $Y_1$ and $Y_2$ is at the mouth.

\vskip5pt

Otherwise, both $Y_1$ and $Y_2$ are at the mouth. Thus $Y_1\ncong Y_2$ (otherwise, $Y_1$ is not the at mouth).
Since the component $C$ where $X_1$ lies in is a characteristic component, it follows from Proposition \ref{formofcomponent} that $C$ has to be of type $\Bbb Z\vec{A_3}$. By  Lemma \ref{for2cycles},  $\mathcal D^b(A)\cong \mathcal D^b(k\vec{A_3})$. Thus, by [AH], the quiver $Q$ of $A$ is just of type $A_3,$ which contradicts the assumption.

\vskip5pt

Thus, at most one of $Y_1$ and $Y_2$ is at the mouth, while both ${\rm Cone}(f_1)$ and ${\rm Cone}(g_1)$ are at the mouth. It follows that
$${\rm Cone}(h) = Y_1[1]\oplus Y_2[1]\ncong {\rm Cone}(g_1f_1) = {\rm Cone}(f_1)\oplus{\rm Cone}(g_1).$$
By ${\rm(E2)}$, $S(X_1) \cong X_2[m_1]$ and
$S^{-1}(X_1) \cong X_2[-m_2]$ for some integers $m_1$ and $m_2$. Thus $\tau^{-1}X_1 \cong X_2[1-m_2]$, and both $g_1f_1: X_1\longrightarrow \tau^{-1}X_1$ and $h: X_1\longrightarrow S(X_1)$ are in some components
of the complex $\Hom^{\bullet}(X_1, X_2)$. Note that $g_1f_1$ and $h$ can not be linear dependent, even if  they are in the same components
of $\Hom^{\bullet}(X_1, X_2)$: otherwise, by the distinguished triangles
$$\xymatrix@R=0.4cm{X_1\ar[r]^-h & S(X_1)\ar[r] & Y_1[1]\oplus Y_2[1] \ar[r] & X_1[1]\\
X_1 \ar[r]^-{g_1f_1} & \tau^{-1} X_1\ar[r] & {\rm Cone}(f_1)\oplus {\rm Cone}(g_1)\ar[r] & X_1[1]}$$
we get a contradiction ${\rm Cone}(h)\cong {\rm Cone}(g_1f_1)$. Write $\dim_k\Hom^{\bullet}(X_1, X_2)$ for $\sum\limits_{i\in\Bbb Z}\dim_k\Hom(X_1, X_2[i])$. Thus $\dim_k\Hom^{\bullet}(X_1, X_2)\ge 2$, and then
$$\dim_k\Hom^{\bullet}(X_1, X_1) = \dim_k\Hom^{\bullet}(X_1, S(X_1)) = \dim_k\Hom^{\bullet}(X_1, X_2) \ge 2$$
which contradicts $(\rm E1)$. This proves that $X_1$ is at the mouth.

\vskip5pt

Hence $(X_1, X_2) \approx (X_1, S(X_1))\approx (X_1, \tau X_1)$, with $X_1$ at the mouth of the characteristic component $C$,
say, of {\rm AG}-invariant $(n, m)$. Finally, we show that $n = 2$. In fact,  $$\tau^2 X_1 = S^{2}(X_1)[-2] \cong S(X_2[m_1-2])\cong X_1[m_1+m_2-2].$$
By the definition of the {\rm AG}-invariant, $n = 2$ or $n = 1$. We claim that $n = 2$. Otherwise, $(1, m)$ is the {\rm AG}-invariant, and thus $\tau X_1\cong X_1[m-1]$ and $S(X_1) \cong X_1[m]$. Then
$$\Hom_{K^{b}(A\mbox{-}{\rm proj})}(X_1, X_1[m]) \cong \Hom_{K^{b}(A\mbox{-}{\rm proj})}(X_1, S(X_1))\cong \End(X_1)\ne 0.$$ By  $(\rm E1)$, $m = 0$ and $S(X_1)\cong X_1$.
Then $\dim_k\End(X_1) = 2$,  by Lemma \ref{rv} (or by Theorem \ref{morphismbetweenforbiddenthreads}). This contradicts $(\rm E1)$.
This completes the proof. \end{proof}

\subsection{Proof of Theorem \ref{main}}

$(1)$ \ Let $X$ be an indecomposable object at the mouth of a characteristic component $C$ with {\rm AG}-invariant $(n, m)$. Thus, $n$ is the smallest positive integer and $m$ is the unique integer, such that $\tau^n Y \cong Y[m-n]$ for an arbitrary indecomposable object $Y$ at the mouth of $C$.

\vskip5pt

First, assume that $n\ge 2$. Then
$S(\tau^{i}X) \cong \tau^{i+1}X[1]$ for $0\le i\le n-2$, and $S(\tau^{n-1}X) \cong \tau^{n}X[1] \cong X[m-n+1]$.
Thus $(X, \tau X, \cdots, \tau^{n-1}X)$ satisfies ${\rm (E2)}$.

\vskip5pt

To check that $(X, \tau X, \cdots, \tau^{n-1} X)$ satisfies ${\rm (E3')}$, assume $n\ge 3$.
It suffices to show $\Hom(X, \tau^{i}X[j]) = 0$ for $2\le i\le n-1$ and for all $j\in \Bbb Z$. Otherwise, since $\tau^{i}X[j]$ is at the mouth of characteristic component $C[j]$,
it follows from Theorem \ref{morphismbetweenforbiddenthreads} that
either $\tau^{i}X[j]\cong X$ or $\tau^{i}X[j]\cong S(X).$ In the first case $\tau^{i}X\cong X[-j]$.
Since $i<n,$ this contradicts the definition of the {\rm AG}-invariant.
In the second case $\tau^{i-1}X \cong X[1-j]$, again a contradiction for the same reason.

\vskip5pt

Now we prove that $(X, \tau X, \cdots, \tau^{n-1} X)$ satisfies ${\rm (E1')}$. Since $X\ncong S(X)$ (otherwise $\tau X = X[-1]$, which contradicts $n\ge 2$),
it follows from Theorem \ref{morphismbetweenforbiddenthreads} that $\End(X)=k$.
Also, $\Hom(X, X[j]) =0$ for $j\ne 0$.  Otherwise,
by Theorem \ref{morphismbetweenforbiddenthreads},  either $X[j]\cong X$ or $X[j]\cong S(X).$
The first case is impossible, since $K^b(A\mbox{-}{\rm proj})$ is homotopy-like. In the second case, we have
$\tau X \cong X[j-1]$, which contradicts $n\ge 2.$ This shows $\Hom^{\bullet}(X, X)=k$, i.e., the condition ${\rm (E1')}$ holds.

\vskip5pt

Thus, if $n\ge 2$ then $(X, \tau X, \cdots, \tau^{n-1}X)$ is an exceptional $n$-cycle.

\vskip5pt

Next, assume $n =1$.  Since $(1, m)$ is the {\rm AG}-invariant of $C$, $S(X) = \tau X[1] \cong X[m-n+1] = X[m]$, i.e., $X$ is an $m$-Calabi-Yau object. To say that $X$ is an exceptional $1$-cycle, it remains to show
$\Hom^{\bullet}(X, X)=k\oplus k[-m]$. We divide into two cases.

If $m = 0$, then $S(X) \cong X$, and then by Theorem \ref{morphismbetweenforbiddenthreads}, $\End(X)=k\oplus k$. For $j\ne 0$, we have $X[j]\ncong X$ and $X[j]\ncong S(X)$,
then by Theorem \ref{morphismbetweenforbiddenthreads}, $\Hom_{K^b(A\mbox{-}{\rm proj})}(X, X[j])= 0$. Thus $\Hom^{\bullet}(X, X)=k\oplus k$.

If $m\ne 0$, then $X\ncong X[m]$, i.e., $X\ncong S(X)$, and then $\End(X) = k$ by Theorem \ref{morphismbetweenforbiddenthreads}. Also,
$$\Hom_{K^b(A\mbox{-}{\rm proj})}(X, X[m])\cong\Hom_{K^b(A\mbox{-}{\rm proj})}(X, S(X))\cong \End(X) = k.$$
For $j\ne 0$ and $j\ne m$, one has $X[j]\ncong X$ and $X[j]\ncong X[m]\cong S(X)$. It follows from Theorem \ref{morphismbetweenforbiddenthreads} that
$\Hom_{K^b(A\mbox{-}{\rm proj})}(X, X[j]) = 0$. Thus $\Hom^{\bullet}(X, X)=k\oplus k[-m]$.

This shows that if $n =1$ then $X$ is an exceptional $1$-cycle.

\vskip5pt

Now, we prove the uniqueness of exceptional cycle in the characteristic component $C$, up to $\approx$.
For this, assume that $(E_1, \cdots, E_l)$ is an arbitrary exceptional cycle in $C$ with $l\ge 1$ (thus each $E_i\in C$). Then $E_1$ is at the mouth: if $l\ne 2$
then this follows from Theorem \ref{indsphericalmouth}$(2)$, and if $l = 2$ then this follows from Lemma \ref{2cyclefornonA3}.
By what we have proved, $(E_1, \tau E_1, \cdots, \tau^{n-1}E_1)$ is also an exceptional cycle. It follows from Lemma \ref{shift}$(1)$ that
$l = n$ and $(E_1, \cdots, E_l)\approx (E_1, \tau E_1, \cdots, \tau^{n-1}E_1)$.

\vskip5pt

It remains to prove that $(E_1, \tau E_1, \cdots, \tau^{n-1}E_1)\approx (X, \tau X, \cdots, \tau^{n-1}X),$ where $X$ is an arbitrary indecomposable object at the mouth of $C$.
By Proposition \ref{formofcomponent}, $C$ is one of the form $$\Bbb Z\vec{A_n} \ \ (n\ge 2), \ \ \ \ \Bbb Z\vec{A_{\infty}}, \ \ \ \ \Bbb Z\vec{A_{\infty}}/\langle \tau^n\rangle \ \ (n\ge 1).$$
If $C$ is of the form $\Bbb Z\vec{A_{\infty}}$ or $\Bbb Z\vec{A_{\infty}}/\langle \tau^n\rangle \ (n\ge 1)$, then $E_1 = \tau^t X$ for some integer $t$, since both $X$ and $E_1$ are at the mouth of $C$. Write $t=qn+r$, where $q$ and $r$ are integers  with $0\le r \le n-1$. Then $E_1=\tau^{t}X=\tau^{qn+r}X =\tau^{r}X[q(m-n)]$, and thus by Lemma \ref{shift}$(1)$ one has
$$(E_1, \tau E_1, \cdots, \tau^{n-1}E_1)\approx (\tau^{r}X, \tau E_1, \cdots, \tau^{n-1}E_1)\approx (X, \tau X, \cdots, \tau^{r}X, \cdots, \tau^{n-1}X).$$
If $C$ is of the form $\Bbb Z\vec{A_n}$, then by Fact \ref{an}, $E_1 = \tau^t (X)[s]$ for some integers $t$ and $s$, since both $X$ and $E_1$ are at the mouth of $C$. Again write $t=qn+r$, where $q$ and $r$ are integers  with $0\le r \le n-1$.
By the same argument as above one has
$$(E_1, \tau E_1, \cdots, \tau^{n-1}E_1)\approx (\tau^{r}X, \tau E_1, \cdots, \tau^{n-1}E_1)\approx (X, \tau X, \cdots, \tau^{r}X, \cdots, \tau^{n-1}X).$$
This completes the proof of assertion $(1)$.

\vskip5pt

$(2)$ \ Let $(E_1, \cdots, E_n)$ be an exceptional cycle in $K^b(A\mbox{-}{\rm proj})$ with $n\ge 2$.
We want to prove that there is a characteristic component $C$ of AG-invariant $(n, m)$ such that $E_1$ is at the mouth of $C$ and
$(E_1, \cdots, E_n)\approx (E_1, \tau E_1, \cdots, \tau^{n-1}E_1)$.

\vskip5pt

First, $E_1$ is a string complex. Otherwise, $E_1$ is a band complex, and then $\tau E_1 = E_1$ (cf. Theorem 2.4$(2)$);
and hence by $(\rm E1)$ we get a contradiction
$0 = \Hom(E_1, E_1[1])=\Hom(E_1, \tau E_1[1]) = \Hom(E_1, S(E_1))\cong \Hom(E_1, E_1)^*.$

\vskip5pt

Second, $E_1$ is at the mouth. In fact, for $n = 2$ this follows from Lemma \ref{2cyclefornonA3}, and for $n \ge 3$ this follows from Theorem \ref{indsphericalmouth}$(2)$.

\vskip5pt

Thus $E_1$ is a string complex at the mouth. By definition $E_1$ is in a characteristic component $C$, say with
AG-invariant $(n', m)$. By $(1)$, $(E_1, \tau E_1, \cdots, \tau^{n'-1}E_1)$ is also an exceptional cycle. By Lemma \ref{shift}$(1)$ we have
$n' = n$ and $(E_1, \cdots, E_n)\approx (E_1, \tau E_1, \cdots, \tau^{n-1}E_1)$.
This completes the proof of assertion $(2)$.

\vskip5pt

$(3)$ \ By $(1)$ and $(2)$, any object in exceptional $n$-cycles with $n\ge 2$ is indecomposable and at the mouth.
By Lemma \ref{indec}, $K^b(A\mbox{-}{\rm proj})$ is indecomposable; since $K^b(A\mbox{-}{\rm proj})$ is homotopy-like,
by Theorem \ref{indsphericalmouth}$(1)$,  any exceptional $1$-cycle in $K^b(A\mbox{-}{\rm proj})$ is indecomposable and at the mouth.

\vskip5pt

Let $E$ be a string complex. If $E$ is at the mouth of a characteristic component of  {\rm AG}-invariant $(1, d)$, then by $(1)$, $E$ is an exceptional $1$-cycle. Conversely,
if $E$ is an exceptional $1$-cycle, such that $E$ is $d$-Calabi-Yau, then $E$ is at the mouth. Hence $E$ is in a characteristic component $C$ of $K^b(A\mbox{-}{\rm proj})$, say with {\rm AG}-invariant $(n, m)$. Since $S(E)\cong E[d]$, $\tau(E)\cong E[d-1]$.
By the definition of {\rm AG}-invariant one has $n =1$ and $m = d$.

\vskip5pt

If $E$ is a band complex which is an exceptional $1$-cycle, then $E$ is at the mouth, and hence $E$ is at the mouth of a homogeneous tube (cf. Theorem \ref{2terms}$(2)$).
\hfill $\square$

\section{\bf Examples}\begin{Ex} \ Let $A=k[x]/\langle x^2\rangle$. Since $A$ is symmetric,
$K^b(A\mbox{-}{\rm proj})$ is a $0$-Calabi-Yau category $($in the sense of {\rm [K])}, and hence each object is $0$-Calabi-Yau.
The indecomposable objects in $K^b(A\mbox{-}{\rm proj})$ are of form $X_l[t]$ with $t\in\mathbb{Z}$ and $l\geq0$, where $X_l$ is the following complex with each differential given by multiplication by $x:$
$$\xymatrix{
0 \ar[r] & A=X_l^{-l}\ar[r]^-x& A\ar[r]^-x & \cdots \ar[r]^-x &  A=X_l^0 \ar[r]  &  0.}$$
The Auslander-Reiten quiver of $K^b(A\mbox{-}{\rm proj})$ has a unique component of the type $\Bbb ZA_\infty$, which is a characteristic component with ${\rm AG}$ invariant $(1,0)$, as given below.
$$
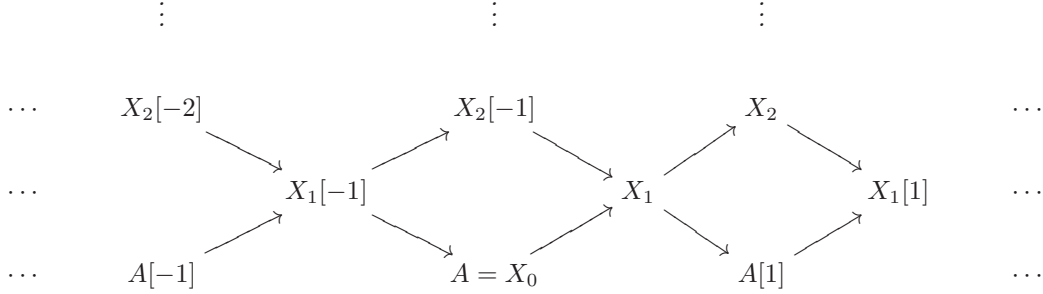
\xymatrix@R=0.5cm{
  &  \vdots &  &  \vdots  &  &  \vdots  &  &  \\
\cdots   &  X_2[-2] \ar[rd]&  &  X_2[-1]\ar[rd]  &  &  X_2 \ar[rd] &  &  \cdots\\
\cdots &  &  X_1[-1]\ar[ru]\ar[rd] &    &  X_1 \ar[ru]\ar[rd]&    & X_1[1] &  \cdots\\
\cdots   &  A[-1]\ar[ru] &  &  A=X_0 \ar[ru]  &  &  A[1] \ar[ru] &  &  \cdots
}$$
\vskip5pt \noindent All the exceptional $1$-cycles are $A[t]$ with $t\in \mathbb{Z}$. There are no exceptional $n$-cycle with $n\geq 2$.
One can see this directly as follows. If $l\ge 1$, then one constructs a chain map $f: X_l\rightarrow X_l[1]$ with all non-zero component being identity.
Then $f$ is not null-homotopic. Since $X_l\ncong X_l[1]$, $f$ can not be an isomorphism, and hence $\Hom^\bullet(X_l, X_l)\ne k$.
\end{Ex}

\vskip5pt

\begin{Ex} \label{bandatmouth} \ Theorem {\rm\ref{main}$(4)$} assert that
a band complex which is an exceptional $1$-cycle is at the mouth.
However, a band complex at the mouth is not necessarily an exceptional $1$-cycle.

\vskip5pt

$(1)$ \ We take an example from {\rm [ALP]}. Let $A=kQ/I$, where $Q$ is the quiver
$$
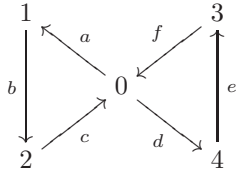
\xymatrix@R=0.5cm{
1 \ar[dd]_b &  &  3\ar[ld]_f\\
  &  0 \ar[lu]_a \ar[rd]_d &\\
2 \ar[ru]_c  &   &  4\ar[uu]_e
}$$
and $I=\langle ba, cb, ac, ed, fe, df\rangle$. Consider the homotopy band $v=d^{-1}e^{-1}f^{-1}cba$ and
the band complex $P_{m, v, \lambda}$, where $m\in\mathbb{Z}$ and $\lambda\in k^{*}$. Then
$P_{m, v, \lambda}$ is at the mouth of a homogeneous tube, and
$\Hom_{K^b(A\mbox{-}{\rm proj})}(P_{m, v, \lambda},P_{m, v, \lambda}[3])\ne 0$ $({\rm [ALP, 5.17]})$.
Thus $P_{m, v, \lambda}$ is not an exceptional $1$-cycle.

\vskip5pt

$(2)$ \ Let $A=k(\xymatrix{1\ar@<.5ex>[r]^{\alpha}\ar@<-.5ex>[r]_{\beta}&2})$ be the Kronecker algebra.
Consider the homotopy band $w=\beta^{-1}\alpha$ and
the band complex $P_{w, \lambda}$ with $\lambda\in k^{*}$: \ $0\longrightarrow P(2) \stackrel {\alpha + \lambda \beta}\longrightarrow P(1) \longrightarrow 0$. Then
$P_{m, w, \lambda}$ is at the mouth of a homogeneous tube, and $P_{w, \lambda}$ is an exceptional $1$-cycle.

\end{Ex}

\vskip10pt

Peng Guo, \ guigui91$\symbol{64}$126.com

Pu Zhang, \ pzhang$\symbol{64}$sjtu.edu.cn

School of Mathematical Sciences, \ Shanghai Jiao Tong University

Shanghai 200240, P. R.\ China


\vskip10pt



\begin{thebibliography}{99}

\bibitem[APS]{APS} C. Amiot, P. Plamondon, S. Schroll, A complete derived invariant for gentle algebras via winding numbers and arf invariants, arXiv: 1904.02555v3 [math.RT]

\bibitem[AG]{AG} K. K. Arnesen, Y. Grimeland, The Auslander-Reiten components of $K^b({\rm proj}\Lambda)$ for a cluster-tilted algebra of type $\tilde{A}$, J. Algebra Appl. 14(1)(2015), 1550005.

\bibitem[ALP]{ALP} K. K. Arnesen, R. Laking, D. Pauksztello, Morphisms between indecomposable complexes in the bounded derived category of a gentle algebra, J. Algebra 467(2016), 1-46.

\bibitem[ABCP]{ABCP} I. Assem, T. Br\"{u}stle, G. Charbonneau-Jodoin, P.-G. Plamondon, Gentle algebras arising from surface triangulations, Algebra Number Theory 4(2)(2010), 201-229.

\bibitem[AH]{AH} I. Assem, D. Happel, Generalized tilted algebras of type $A_n$, Comm. Algebra 9(1981), 2101-2125.

\bibitem[ASS]{ASS} I. Assem, D. Simson, A. Skowro\'nski,
Elements of the representation theory of associative algebras,
Vol.1: Techniques of representation theory,
Lond. Math. Soc. Students Texts 65,
Cambridge University Press, 2006.

\bibitem[AS]{AS} I. Assem, A. Skowro\'nski, Iterated tilted algebras of type $\widetilde{A_n}$, Math.
Z. 195(2)(1987), 269-290.

\bibitem[AAG]{AAG} D. Avella-Alaminos, C. Geiss, Combinatorial derived invariants for gentle algebras, J. Pure Appl. Algebra 212(1)(2008), 228-243.

\bibitem[BM]{BM} V. Bekkert, H. A. Merklen, Indecomposables in derived categories of gentle algebras, Algebr. Represent. Theory 6(3)(2003), 285-302.

\bibitem[B]{B} G. Bobi\'nski, The almost split triangles for perfect complexes over gentle algebras, J. Pure Appl. Algebra 215(4)(2011), 642-654.

\bibitem[BB]{BB} G. Bobi\'nski, A. B. Buan, The algebras derived equivalent to gentle cluster tilted algebras,
J. Algebra Appl. 11(1)(2012), 1250012.

\bibitem[BGS]{BGS} G. Bobi\'nski, C. Geiss, A. Skowro\'nski, Classification of discrete derived categories, Cent. Eur. J. Math. 2(1)(2004), 19-49.

\bibitem[Boc]{Boc} R. Bocklandt, Graded Calabi Yau algebras of  dimension
$3$, with an appendix ``The signs of Serre functor" by M. Van den
Bergh, J. Pure Appl. Algebra {212}(1)(2008), 14-32.

\bibitem[Br]{Br} S. Brenner, On the kernel of an irreducible map, Linear Algebra and its Applications 365(2003), 91-97.

\bibitem[BPP]{BPP} N. Broomhead, D. Pauksztello, D. Ploog, Discrete derived categories I: homomorphisms, autoequivalences and t-structures, Math. Z. 285(1-2)(2017), 39-89.

\bibitem[BR]{BR} M. C. R. Butler, C. M. Ringel, Auslander-Reiten sequences with few middle terms and applications to string algebras, Comm. Algebra 15(1-2)(1987), 145-179.

\bibitem[CP]{CP} R. Coelho Sim\~oes,  D. Pauksztello, Torsion pairs in a triangulated category generated by a spherical object,  J. Algebra 448(2016), 1-47.

\bibitem[CZ]{CZ}C. Cibils,  P. Zhang,  Calabi-Yau objects in triangulated categories, Trans. Amer. Math. Soc. 361(12)(2009), 6501-6519.

\bibitem[GR]{GR} C. Geiss, I. Reiten, Gentle algebras are Gorenstein, in Representations of algebras and related topics, Fields Inst. Commun., vol. 45, Amer. Math. Soc., Providence, RI, 2005, pp. 129-133.

\bibitem[GZ]{GZ} P. Guo, P. Zhang, Exceptional cycles in the bounded derived categories of quivers, Acta Math. Sinica (English Series) (to appear).

\bibitem[H1]{H1} D. Happel, Triangulated categories in
representation theory of finite dimensional algebras, London Math.
Soc. Lecture Notes Ser. 119, Cambridge Uni. Press,
1988.

\bibitem[H2]{H2} D. Happel, On Gorenstein algebras, in Representation Theory of Finite Groups and Finite-Dimensional Algebra, Progr. Math. 95, Birkha\"user, Basel, 1991, 389-404.

\bibitem[H3]{H3} D. Happel, Auslander-Reiten triangles in derived categories of finite-dimensional algebras, Proc. Amer. Math. Soc. 112(3)(1991), 641-648.

\bibitem[HPR]{HPR} D. Happel, U. Preiser, C. M. Ringel, Vinberg's characterization of Dykin diagrams using subadditive functions with application to DTr-periodic modules, In: Representation theory II (Proc. Second. Internat. Conf., Carleton Univ., Ottawa, 1979), 280-294, Lecture Notes in Math. vol. 832, Springer, Berlin, 1980.

\bibitem[HKR]{HKR} D. Happel, B. Keller, I. Reiten, Bounded derived categories and repetitive algebras, J. Algebra 319(4)(2008), 1611-1635.

\bibitem[HKP1]{HKP1} A. Hochenegger, M. Kalck, D. Ploog, Spherical subcategories in algebraic geometry, Math. Nachr. 289(11-12)(2016), 1450-1465.

\bibitem[HKP2]{HKP2} A. Hochenegger, M. Kalck, D. Ploog, Spherical subcategories in representation theory, Math. Z. 291(2019), 113-147.

\bibitem[K]{K} B. Keller, On triangulated orbit categories, Documenta Math. 10(2005), 551-581.

\bibitem[KYZ]{KYZ} B. Keller, D. Yang, G. Zhou, The Hall algebra of a spherical object, J. Lond. Math. Soc.(2)80(3)(2009), 771-784.

\bibitem [M]{M} H. Meltzer, Tubular mutations, Colloq. Math. 74(2)(1997), 267-274.

\bibitem[N]{N} A. Neeman, Triangulated categories, Ann. Math. Studies 148, Princeton University Press, 2001.

\bibitem[O]{O} S. Opper, On auto-equivalences and complete derived invariants of gentle algebras, arXiv: 1904.04859v1 [math.RT]

\bibitem[OPS]{OPS} S. Opper, P. Plamondon, S. Schroll, A geometry model for the derived category of gentle algebras, arXiv: 1801.09659v4 [math.RT].

\bibitem[PS]{PS} Z. Pogorzaly, A. Skowro\'nski, Selfinjective biserial standard algebras, J. Algebra 138(1991), 491-504.

\bibitem[RV]{RV} I. Reiten, M. Van den Bergh, Noether hereditary abelian categories satisfying Serre functor, J. Amer. Math. Soc. 15(2)(2002), 295-366.

\bibitem[Rm]{Rm} C. Riedtmann, Algebren, Darstellungsk\"{o}cher, \"{U}berlagerungen und zur\"{u}ck, Comment. Math. Helv. 55(2)(1980), 199-224.

\bibitem[Ro]{Ro} R. Rouquier, Dimensions of triangulated categories, J. K-Theory 1(2008), 193-256.

\bibitem[Sch]{Sch} S. Scherotzke, Finite and bounded Auslander-Reiten components in the derived category, J. Pure Appl. Algebra 215(2011), 232-241.

\bibitem[SZ]{SZ} J. Schr\"{o}er, A. Zimmermann, Stable endomorphism algebras of modules over speicial biserial algebras, Math. Z. 244(2003), 515-530.

\bibitem[ST]{ST} P. Seidel, R. Thomas, Braid group actions on derived categories of coherent sheaves, Duke Math. J. 108 (1)(2001), 37-108.

\bibitem[V]{V} D. Vossieck, The algebras with discrete derived category, J. Algebra 243(1)(2001), 168-176.

\bibitem[XZ]{XZ} J. Xiao, B. Zhu, Locally finite triangulated categories, J. Algebra 290(2)(2005), 473-490.

\end{thebibliography}
\end{document}